\documentclass[11pt]{article}
\usepackage{etex}
\usepackage[all]{xy}
\usepackage[utf8]{inputenc}
\usepackage{url}
\usepackage{todonotes}

\usepackage{amsfonts}
\usepackage{amsthm}
\usepackage{enumerate}
\usepackage{graphicx}
\usepackage{mathrsfs}
\usepackage{bm}
\usepackage{cite}
\usepackage{amssymb,amsmath} 
\usepackage{makecell}
\usepackage{tikz}
\usepackage{pgf}
\usepackage{tikz}
\usetikzlibrary{patterns}
\usepackage{pgffor}
\usepackage{pgfcalendar}
\usepackage{pgfpages}
\usepackage{shuffle,yfonts}
\usepackage{mathtools}
\DeclareFontFamily{U}{shuffle}{}
\DeclareFontShape{U}{shuffle}{m}{n}{ <-8>shuffle7 <8->shuffle10}{}

\newcommand{\bfmu}{{\boldsymbol{\mu}}}

\newcommand{\bfk}{{\boldsymbol{\sl{k}}}}

\newcommand{\bfn}{{\boldsymbol{\sl{n}}}}

\newcommand{\bfx}{{\boldsymbol{\sl{x}}}}
\newcommand{\bfz}{{\boldsymbol{\sl{z}}}}

\newcommand{\de}{\mathrm{d}}
 \allowdisplaybreaks
\usetikzlibrary{arrows,shapes,chains}
 \allowdisplaybreaks

\catcode`!=11
\let\!int\int \def\int{\displaystyle\!int}
\let\!lim\lim \def\lim{\displaystyle\!lim}
\let\!sum\sum \def\sum{\displaystyle\!sum}
\let\!sup\sup \def\sup{\displaystyle\!sup}
\let\!inf\inf \def\inf{\displaystyle\!inf}
\let\!cap\cap \def\cap{\displaystyle\!cap}
\let\!max\max \def\max{\displaystyle\!max}
\let\!min\min \def\min{\displaystyle\!min}
\let\!frac\frac \def\frac{\displaystyle\!frac}
\catcode`!=12

\let\oldsection\section
\renewcommand\section{\setcounter{equation}{0}\oldsection}

\allowdisplaybreaks

\DeclareMathOperator*{\dep}{dep}
\DeclareMathOperator{\Li}{Li}

\def\N{\mathbb{N}}
\def\Z{\mathbb{Z}}
\def\Q{\mathbb{Q}}
\def\CC{\mathbb{C}}

\def\ze{\zeta}

\theoremstyle{plain}
\newtheorem{thm}{Theorem}[section]
\newtheorem{lem}[thm]{Lemma}
\newtheorem{cor}[thm]{Corollary}
\newtheorem{con}[thm]{Conjecture}
\newtheorem{pro}[thm]{Proposition}
\theoremstyle{definition}
\newtheorem{defn}{Definition}[section]

\newtheorem{exa}[thm]{Example}

\begin{document}
\title{\bf Residue Theorem, Regularization and Parity Theorem}
\author{{{Jia Li${}^{a,}$\thanks{Email: jialimath001@pku.org.cn}}\quad and \quad Ce Xu${}^{b,}$\thanks{Email: cexu2020@ahnu.edu.cn}}\\[1mm]
\small a. School of Mathematical Sciences, Peking University, Beijing 100000, P.R. China\\
\small b. School of Mathematics and Statistics, Anhui Normal University, Wuhu 241002, P.R. China}

\date{}
\maketitle
\noindent{\bf Abstract.} In this paper, we employ contour integration and residue calculus to derive explicit parity formulas for (cyclotomic) multiple zeta values (MZVs). A key innovation lies in applying double shuffle regularization to the contour integrals, which leads to two distinct regularized parity formulas-one via shuffle and one via stuffle regularization. Notably, this demonstrates for the first time that the contour integral method can be extended to the regularized setting (including the case
$k_r=1$), thereby overcoming a limitation of previous approaches. Our results not only provide explicit parity relations at arbitrary depths but also lay the groundwork for extending this technique to other variants of multiple zeta values.

\medskip

\medskip


\section{Introduction}

For a multi-index $\bfk:=(k_1,\ldots, k_r)\in(\Z_{>0})^r$, we call it \emph{positive} multi-index. If, in addition, $k_r>1$, $\bfk$ is called \emph{admissible}. We put
\begin{equation*}
 |\bfk|:=k_1+\cdots+k_r,\quad \dep(\bfk):=r,
\end{equation*}
and call them the \emph{weight} and the \emph{depth} of $\bfk$, respectively.
As a convention, we denote by $\{m\}_r$ the sequence of $m$'s with $r$ repetitions.

For an admissible multi-index $\bfk:=(k_1,\ldots, k_r)$, the classical \emph{multiple zeta values} (MZVs) are defined by (\cite{H1992,DZ1994})
\begin{align}
\zeta(\bfk)\equiv \zeta(k_1,\ldots,k_r):=\sum_{0<n_1<\cdots<n_r} \frac{1}{n_1^{k_1}\cdots n_r^{k_r}}.
\end{align}
Another object frequently studied alongside classical multiple zeta values is known as the \emph{multiple zeta star values} (MZSVs), which are defined as
\begin{align}
\zeta^\star(\bfk)\equiv \zeta^\star(k_1,\ldots,k_r):=\sum_{0<n_1\leqslant \cdots\leqslant n_r} \frac{1}{n_1^{k_1}\cdots n_r^{k_r}}.
\end{align}
The concept of multiple zeta values was independently introduced in the early 1990s by Hoffman \cite{H1992} and Zagier \cite{DZ1994}. Owing to their deep connections with various mathematical and physical disciplines-such as knot theory, algebraic geometry, and theoretical physics-the study of multiple zeta values has attracted sustained interest from numerous mathematicians and physicists. After more than three decades of development, the field has accumulated a wealth of research results. For a comprehensive overview of advances prior to 2016, readers are referred to Zhao's authoritative monograph \cite{Z2016}. In addition, over the years of continued research, various generalizations and variants of multiple zeta values have been introduced and studied. These include, for instance, alternating multiple zeta values \cite{BorweinBrBr1997}, multiple $t$-values \cite{Charlton-MathAnn2025,CharltonHoffman-MathZ2025,H2019,Murakami2021}, multiple $T$-values \cite{KanekoTs2019}, Mordell-Tornheim zeta functions \cite{DSS2025}, among others. For relevant references, see, e.g., \cite{KanekoYa2018,Li2019,Li2024,XuZhao2020d} and the literature therein. 

Among various research questions on multiple zeta values and related variants, the study of parity is a relatively important issue. In \cite{Borwein-Girgensohn-1996}, Borwein and Girgensohn conjectured the following fascinating result which is called the parity result or the parity conjecture for multiple zeta values:

\begin{con}[Borwein--Girgensohn \cite{Borwein-Girgensohn-1996}, 1996]
    For $r\in \Z_{>1}$ and $\bfk=(k_1,\ldots,k_r)\in (\Z_{>0})^r$ with $k_r>1$, $\ze(k_1,\ldots,k_r)$ can be expressed in terms of lower depth multiple zeta values when its depth and weight are of different parity.
\end{con}
The case of depth 2 has been already considered by Euler, and the case of depth 3 was proved by Borwein and Girgensohn in \cite{Borwein-Girgensohn-1996}. Ihara-Kaneko-Zagier \cite{IKZ2006} gave the proof in the general case. Tsumura \cite{Tsu-2004} gave another proof of this result with a different method. Further, Tsumura \cite{Tsu-2007} proved that the multiple $L$-value of conductor 4 can be expressed in terms of lower depth multiple $L$-values under the condition on the parity of its depth and weight. Regrettably, none of the aforementioned proofs were able to provide a general explicit formula of parity result. When it comes to the study of explicit formulas for the parity relations of multiple zeta values and their variants, Panzer's 2017 work on functional equations for multiple polylogarithms \cite{Panzer2017} undoubtedly stands out as a highly significant contribution. Not only does this result provide parity relations for cyclotomic multiple zeta values at arbitrary depths, but it also delivers a computer program to compute the functional equations. Moreover, Panzer further derives explicit formulas for depths 2 and 3. The parity results for multiple polylogarithms established by Panzer are stated as follows \cite[Theorem 1.3]{Panzer2017}.

\begin{thm}[Panzer \cite{Panzer2017}, 2017]
    For all $r\in \Z_{>1}$ and $\bfn=(n_1,\ldots,n_r)\in(\Z_{>0})^r$, the function
\begin{align*}
\Li_{\bfn}(z_1,z_2,\ldots,z_r)-(-1)^{n_1+\cdots+n_r-r}\Li_{\bfn}(1/z_1,1/z_2,\ldots,1/z_r)
\end{align*}
is of depth at most $r-1$, meaning that it can be written as a $\Q$-linear combination of the functions
\begin{align*}
(2\pi i)^{k_0} \prod\limits_{i=1}^d \log^{k_i}(-z_i\cdots z_d)\prod\limits_{i=1}^s \Li_{{\bfn}^{(i)}}({\bfz}^{(i)}),
\end{align*}
where the indices ${\bfn}^{(i)}\in \N^{d_i}$ have total depth $d_1+\cdots+d_s<d$ and preserve the weight $|\bfk|+\sum_{i=1}^s |\bfn^{(i)}|=|\bfn|$. Each of the arguments $\bfz_j^{(i)}$ is a consecutive product $z_\mu z_{\mu+1}\cdots z_\nu$ for some $\mu\leqslant \nu$. Here, for any $(k_1,\dotsc,k_r)\in(\Z_{>0})^r$, the classical \emph{multiple polylogarithm function} with $r$-variables is defined by
\begin{align}\label{defn-mpolyf}
\Li_{k_1,\dotsc,k_r}(x_1,\dotsc,x_r):=\sum_{0<n_1<\cdots<n_r} \frac{x_1^{n_1}\dotsm x_r^{n_r}}{n_1^{k_1}\dotsm n_r^{k_r}}
\end{align}
which converges if $|x_j\cdots x_r|<1$ for all $j=1,\dotsc,r$. It can be analytically continued to a multi-valued meromorphic function on $\mathbb{C}^r$ (see \cite{Zhao2007d}).
\end{thm}
However, it is important to note that while Panzer's paper provides a computational algorithm to derive functional equations, it still falls short of delivering an explicit formula for the parity relations of multiple zeta values and multiple polylogarithms at arbitrary depths. Hirose \cite{Hirose2025} recently established an explicit formula for the parity of multiple zeta values by employing the theory of multitangent functions developed by Bouillot \cite{Bouillot2014}. Umezawa \cite{Umezawa2025arxiv} has recently extended Hirose's parity results to the case of multiple polylogarithms.

As is well known, the theory of contour integration and the residue theorem serves as a highly effective method for studying infinite series. In their 1998 paper~\cite{Flajolet-Salvy}, Flajolet and Salvy systematically investigated the parity properties of a class of multiple zeta value variants---called ``Euler sums'' (which can be expressed as integer-linear combinations of multiple zeta values)---using contour integration. Subsequently, many authors have extended their contour integral approach to study parity results for related variants such as multiple $t$-values\cite{H2019}, multiple $T$-values\cite{KanekoTs2019}, and multiple $M$-values\cite{XuZhao2020a}---collectively referred to as ``Euler $T$-sums, Euler $S$-sums, etc.'' (which can be expressed as integer-linear combinations of multiple mixed values). Relevant results can be found in references such as \cite{Xu-Wang2022,XuZhao2020a}.

However, the results in these papers are unable to provide parity formulas for arbitrary depths of multiple zeta values, multiple $t$-values, multiple $T$-values, or multiple $M$-values. Moreover, it is important to note that the contour integration methods used in the above studies cannot handle \emph{regularized} cases (i.e., the situation where $k_r = 1$).

The aim of this paper is to employ the method of contour integration to derive explicit parity formulas for multiple zeta values. Our innovation lies in applying double shuffle regularization to the contour integrals, thereby obtaining two distinct parity formulas: one via shuffle regularization and the other via stuffle regularization. Furthermore, this approach can be extended to investigate parity properties of many other variants of multiple zeta values, such as cyclotomic multiple zeta values.

The structure of this paper is organized as follows.
In Section 2, we first define several classes of finite multiple Hurwitz zeta values with specified summation ranges. We then investigate some fundamental properties satisfied by these finite multiple Hurwitz zeta functions. Subsequently, we present series expansions or Laurent expansions of these finite multiple Hurwitz zeta values at integer points. These formulas play a crucial role in computing residues at poles of the integrand in the contour integrals discussed in later sections.

Section 3 introduces the main approach of this work and provides some necessary lemmas.

In Section 4, we first use the method of contour integration to establish several infinite series identities involving finite multiple Hurwitz zeta values. By regularizing these identities and taking appropriate limits, we derive two types of regularized parity formulas for multiple zeta values of arbitrary depth.

Finally, in the Section 5, we present parity formulas under both shuffle and stuffle regularizations for cyclotomic multiple zeta values, obtained via analogous techniques. Since the procedure closely parallels that of the main text, we omit the detailed derivation here.

\section{Properties and Expansion}

\begin{defn}\label{def2.1}
    For a positive multi-index $\bfk=(k_1,\ldots,k_r)$ , let $m_1,m_2\in\mathbb{Z}\cup\{-\infty,+\infty\}$ satisfy $m_1<m_2$. We define the following notations
    \begin{align*}
        \zeta_{(m_1,m_2)}(\bfk;s)&:=\sum_{m_1<n_1<\cdots<n_r<m_2}\frac{1}{(n_1+s)^{k_1}\cdots (n_r+s)^{k_r}},\\
        \zeta_{(m_1,m_2]}(\bfk;s)&:=\sum_{m_1<n_1<\cdots<n_r\leqslant m_2}\frac{1}{(n_1+s)^{k_1}\cdots (n_r+s)^{k_r}}.
    \end{align*}
    Similarly, we also define the following notations
    \begin{align*}
        \zeta_{(m_1,m_2)}^{\star}(\bfk;s)&=\sum_{m_1<n_1\leqslant\cdots\leqslant n_r<m_2}\frac{1}{(n_1+s)^{k_1}\cdots (n_r+s)^{k_r}},\\
        \zeta_{(m_1,m_2]}^{\star}(\bfk;s)&=\sum_{m_1<n_1\leqslant\cdots\leqslant n_r\leqslant m_2}\frac{1}{(n_1+s)^{k_1}\cdots (n_r+s)^{k_r}}.
    \end{align*}
    To ensure the convergence of the series, when $m_2=+\infty$ (respectively, $m_1=-\infty$), we require $k_r>1$ (respectively, $k_1>1$). In particular, if $s=0$, we denote the above symbol simply by
    $\zeta_{(m_1,m_2)}(\bfk):=\zeta_{(m_1,m_2)}(\bfk;0),\zeta_{(m_1,m_2]}(\bfk):=\zeta_{(m_1,m_2]}(\bfk;0),\zeta_{(m_1,m_2)}^{\star}(\bfk):=\zeta_{(m_1,m_2)}^{\star}(\bfk;0)$ and $\zeta_{(m_1,m_2]}^{\star}(\bfk)=\zeta_{(m_1,m_2]}^{\star}(\bfk;0)$.
\end{defn}

\begin{exa}\label{exa-2.1}
    We can easily see that
   $$\zeta(\bfk;s)=\zeta_{(0,+\infty)}(\bfk;s):=\sum_{0<n_1< \cdots< n_r} \frac{1}{(n_1+s)^{k_1}\cdots (n_r+s)^{k_r}},$$
   and
    $$\zeta^{\star}(\bfk;s)=\zeta^{\star}_{(0,+\infty)}(\bfk;s):=\sum_{0<n_1\leqslant \cdots\leqslant n_r} \frac{1}{(n_1+s)^{k_1}\cdots (n_r+s)^{k_r}}.$$
    In particular, for an admissible multi-index $\bfk:=(k_1,\ldots,k_r)$, we have
    $$\zeta(\bfk)=\zeta_{(0,+\infty)}(\bfk;0)=\sum_{0<n_1< \cdots< n_r} \frac{1}{n_1^{k_1}\cdots n_r^{k_r}}$$
    and
    $$\zeta^{\star}(\bfk)=\zeta^{\star}_{(0,+\infty)}(\bfk;0):=\sum_{0<n_1\leqslant \cdots\leqslant n_r} \frac{1}{n_1^{k_1}\cdots n_r^{k_r}}.$$
\end{exa}

\begin{defn}
    Let $\boldsymbol{y}=(y_1,\ldots,y_r)\in\mathbb{C}^r$ and $0\leqslant i\leqslant j\leqslant r+1$, we define the following notations
    \begin{align*}
        \boldsymbol{y}_{(i,j)}&:=(y_{i+1},y_{i+2},\ldots,y_{j-2},y_{j-1})\in\mathbb{C}^{j-i-1},\\
        \boldsymbol{y}_{[i,j)}&:=(y_{i},y_{i+1},\ldots,y_{j-2},y_{j-1})\in\mathbb{C}^{j-i},\\
        \boldsymbol{y}_{(i,j]}&:=(y_{i+1},y_{i+2},\ldots,y_{j-1},y_{j})\in\mathbb{C}^{j-i},\\
        \boldsymbol{y}_{[i,j]}&:=(y_{i},y_{i+1},\ldots,y_{j-1},y_{j})\in\mathbb{C}^{j-i+1},
    \end{align*}
    and $\overleftarrow{\boldsymbol{y}_{[i,j]}}:=(y_{j},y_{j-1},\ldots,y_{i+1},y_i)$. Similarly, $\overleftarrow{\boldsymbol{y}_{[i,j)}},\overleftarrow{\boldsymbol{y}_{(i,j]}}$ and $\overleftarrow{\boldsymbol{y}_{(i,j)}}$ can be defined in the same way.
\end{defn}
\begin{pro}\label{prop}
    Let $m_1,m_2\in\Z\cup\{-\infty,+\infty\},\ m_1<m_2$ and $n\in\mathbb{Z}$, we have the following identities
    \begin{itemize}
        \item[(1)] (Translation)
        \begin{align*}
        \zeta_{(m_1,m_2)}(k_1,\ldots,k_r;s)=\zeta_{(m_1+n,m_2+n)}(k_1,\ldots,k_r;s-n).
    \end{align*}

    \item[(2)] (Decomposition) If $m_1<n<m_2$, then we have
    \begin{align*}
        &\zeta_{(m_1,m_2)}(k_1,\ldots,k_r;s)\\
        &=\sum_{j=0}^r\zeta_{(m_1,n]}(k_1,\ldots,k_j;s)\zeta_{(n,m_2)}(k_{j+1},\ldots,k_r;s)\\
        &=\sum_{j=0}^r\zeta_{(m_1,n)}(k_1,\ldots,k_j;s)\zeta_{(n,m_2)}(k_{j+1},\ldots,k_r;s)\\
        &\quad+\sum_{j=1}^r\frac{1}{(s+n)^{k_j}}\zeta_{(m_1,n)}(k_1,\ldots,k_{j-1};s)\zeta_{(n,m_2)}(k_{j+1},\ldots,k_r;s).
    \end{align*}

    \item[(3)] (Reflection)
    $$\zeta_{(m_1,m_2)}(k_1,\ldots,k_r;s)=(-1)^{|\bfk|}\zeta_{(-m_2,-m_1)}(k_r,\ldots,k_1;-s).$$

    \item[(4)] (Antipode identity)
    $$\sum_{j=0}^r(-1)^j\zeta^{\star}_{(m_1,m_2)}(k_j,\ldots,k_1;s)\zeta_{(m_1,m_2)}(k_{j+1},\ldots,k_r;s)=0.$$

    \item[(5)] (Truncation) If $0<m_1<m_2$, then we have
    \begin{align*}
        \zeta_{(m_1,m_2)}(k_1,\ldots,k_r;s)=\sum_{j=0}^{r} (-1)^j \zeta^\star_{(0,m_1]}(k_{j},\ldots,k_1;s)\cdot\zeta_{(0,m_2)}(k_{j+1},k_{j+2},\ldots,k_r;s).
    \end{align*}

\item[(6)] (Expansion) If $m_1<m_2\leqslant0$ or $0\leqslant m_1<m_2$, then we have
\begin{align*}
    \zeta_{(m_1,m_2)}(k_1,\ldots,k_r;s)=\sum_{m=0}^{\infty}\left(\sum_{|\bfn|=m}\prod_{l=1}^r\binom{-k_l}{n_l}\zeta_{(m_1,m_2)}(k_1+n_1,\ldots,k_r+n_r)\right)s^m,
\end{align*}
where $\bfn=(n_1,\ldots,n_r)\in(\mathbb{Z}_{\geqslant0})^r$ and $|s|<1$.
    \end{itemize}
\end{pro}

\begin{proof}
(1) Replacing $s$ in the denominator of the finite multiple sum defining $\zeta_{(m_1,m_2)}(\bfk;s)$ by $(n+s-n)$ and performing a direct calculation yields
\begin{align*}
    \zeta_{(m_1,m_2)}(\bfk;s)&=\sum_{m_1<n_1<\cdots<n_r<m_2}\frac{1}{(n_1+s)^{k_1}\cdots (n_r+s)^{k_r}}\\
    &=\sum_{m_1<n_1<\cdots<n_r<m_2}\frac{1}{(n_1+n+s-n)^{k_1}\cdots (n_r+n+s-n)^{k_r}}\\
    &=\sum_{m_1+n<n_1<\cdots<n_r<m_2+n}\frac{1}{(n_1+s-n)^{k_1}\cdots (n_r+s-n)^{k_r}}\\
    &=\zeta_{(m_1+n,m_2+n)}(\bfk;s-n).
\end{align*}

(2) By truncating the summation indices $ m_1 < n_1 < \cdots < n_r < m_2$ at $n$, we obtain
\begin{align*}
    &\zeta_{(m_1,m_2)}(k_1,\ldots,k_r;s)\\
     &=\sum_{m_1<n_1<\cdots<n_r<m_2}\frac{1}{(n_1+s)^{k_1}\cdots(n_r+s)^{k_r}}\\
     &=\sum_{j=0}^r\sum_{m_1<n_1<\cdots<n_j\leqslant n<n_{j+1}<\cdots<n_r<m_2}\frac{1}{(n_1+s)^{k_1}\cdots(n_r+s)^{k_r}}\\
     &=\sum_{j=0}^r\zeta_{(m_1,n]}(k_1,\ldots,k_j;s)\zeta_{(n,m_2)}(k_{j+1},\ldots,k_r;s)\\
     &=\sum_{j=0}^r\zeta_{(m_1,n)}(k_1,\ldots,k_j;s)\zeta_{(n,m_2)}(k_{j+1},\ldots,k_r;s)\\
        &\quad+\sum_{j=1}^r\frac{1}{(s+n)^{k_j}}\zeta_{(m_1,n)}(k_1,\ldots,k_{j-1};s)\zeta_{(n,m_2)}(k_{j+1},\ldots,k_r;s).
\end{align*}

(3) By changing the summation indices $m_1 < n_1 < \dots < n_r < m_2$ to  $-m_1 > -n_1 > \dots > -n_r > -m_2$, we obtain
\begin{align*}
    \zeta_{(m_1,m_2)}(k_1,\ldots,k_r;s)&=\sum_{m_1<n_1<\cdots<n_r<m_2}\frac{1}{(n_1+s)^{k_1}\cdots(n_r+s)^{k_r}}\\
    &=(-1)^{|\bfk|}\sum_{m_1<n_1<\cdots<n_r<m_2}\frac{1}{(-n_1-s)^{k_1}\cdots(-n_r-s)^{k_r}}\\
    &=(-1)^{|\bfk|}\sum_{-m_1>-n_1>\cdots>-n_r>-m_2}\frac{1}{(-n_1-s)^{k_1}\cdots(-n_r-s)^{k_r}}\\
    &=(-1)^{|\bfk|}\sum_{-m_2<n_r'<\cdots<n_1'<-m_1}\frac{1}{(n_1'-s)^{k_1}\cdots(n_r'-s)^{k_r}}\\
    &=(-1)^{|\bfk|}\zeta_{(-m_2,-m_1)}(k_r,\ldots,k_1;-s).
\end{align*}
(4) If $r=1$, then
\begin{align*}
    \zeta_{(m_1,m_2)}(k_1;s)-\zeta^{\star}_{(m_1,m_2)}(k_1;s)=\zeta_{(m_1,m_2)}(k_1;s)-\zeta_{(m_1,m_2)}(k_1;s)=0.
\end{align*}

We assume the cases $r-1$, for $r$, we have
\begin{align*}
    &\sum_{j=0}^r(-1)^j\zeta^{\star}_{(m_1,m_2)}(k_j,\ldots,k_1;s)\zeta_{(m_1,m_2)}(k_{j+1},\ldots,k_r;s)\\
    &=\sum_{j=0}^{r-2}(-1)^j\zeta^{\star}_{(m_1,m_2)}(k_j,\ldots,k_1;s)\zeta_{(m_1,m_2)}(k_{j+1},\ldots,k_r;s)\\
    &\quad+(-1)^{r-1}\zeta^{\star}_{(m_1,m_2)}(k_{r-1},\ldots,k_1;s)\zeta_{(m_1,m_2)}(k_r;s)\\
    &\quad+(-1)^{r}\zeta^{\star}_{(m_1,m_2)}(k_{r},\ldots,k_1;s)\\
    &=\sum_{j=0}^{r-2}(-1)^j\zeta^{\star}_{(m_1,m_2)}(k_j,\ldots,k_1;s)\zeta_{(m_1,m_2)}(k_{j+1},\ldots,k_{r-1};s)\cdot\zeta_{(n_{r-1},m_2)}(k_r;s)\\
    &\quad+(-1)^{r-1}\zeta^{\star}_{(m_1,m_2)}(k_{r-1},\ldots,k_1;s)\zeta_{(m_1,m_2)}(k_r;s)\\
    &\quad+(-1)^{r}\zeta^{\star}_{(m_1,m_2)}(k_{r},\ldots,k_1;s)\\
    &=(-(-1)^{r-1}\zeta^{\star}_{(m_1,m_2)}(k_{r-1},\ldots,k_1;s))\cdot\zeta_{(n_{r-1},m_2)}(k_r;s)\qquad(\text{by\ induction})\\
    &\quad+(-1)^{r-1}\zeta^{\star}_{(m_1,m_2)}(k_{r-1},\ldots,k_1;s)\cdot\zeta_{(m_1,m_2)}(k_r;s)\\
    &\quad+(-1)^{r}\zeta^{\star}_{(m_1,m_2)}(k_{r},\ldots,k_1;s)\\
    &=(-1)^{r-1}\zeta^{\star}_{(m_1,m_2)}(k_{r-1},\ldots,k_1;s)\cdot\zeta_{(m_1,n_{r-1}]}(k_r;s)\\
    &\quad+(-1)^{r}\zeta^{\star}_{(m_1,m_2)}(k_{r},\ldots,k_1;s)\\
    &=(-1)^{r-1}\zeta^{\star}_{(m_1,m_2)}(k_{r},\ldots,k_1;s)\\
    &\quad+(-1)^{r}\zeta^{\star}_{(m_1,m_2)}(k_{r},\ldots,k_1;s)\\
    &=0.
\end{align*}

(5) If $r=1$, then
\begin{align*}
    \zeta_{(m_1,m_2)}(k_1;s)&=\sum_{m_1<n_1<m_2}\frac{1}{(n_1+s)^{k_1}}\\
    &=\sum_{0<n_1<m_2}\frac{1}{(n_1+s)^{k_1}}-\sum_{0<n_1\leqslant m_1}\frac{1}{(n_1+s)^{k_1}}\\
    &=\zeta_{(0,m_2)}(k_1;s)-\zeta^{
    \star}_{(0,m_1]}(k_1;s).
\end{align*}

We assume the cases $r-1$, for $r$, since
 \begin{align*}
     &\zeta_{(0,m_2)}(k_1,\ldots,k_r;s)\\
     &=\sum_{0<n_1<\cdots<n_r<m_2}\frac{1}{(n_1+s)^{k_1}\cdots(n_r+s)^{k_r}}\\
     &=\sum_{j=0}^r\sum_{0<n_1<\cdots<n_j\leqslant m_1<n_{j+1}<\cdots<n_r<m_2}\frac{1}{(n_1+s)^{k_1}\cdots(n_r+s)^{k_r}}\\
     &=\sum_{j=0}^r\zeta_{(0,m_1]}(k_1,\ldots,k_j;s)\zeta_{(m_1,m_2)}(k_{j+1},\ldots,k_r;s)\\
     &=\zeta_{(m_1,m_2)}(k_{1},\ldots,k_r;s)\\
     &\quad+\sum_{j=1}^r\zeta_{(0,m_1]}(k_1,\ldots,k_j;s)\zeta_{(m_1,m_2)}(k_{j+1},\ldots,k_r;s)\\
     &=\zeta_{(m_1,m_2)}(k_{1},\ldots,k_r;s)\qquad(\text{by\ induction})\\
     &\quad+\sum_{j=1}^r\zeta_{(0,m_1]}(k_1,\ldots,k_j;s)\times\left(\sum_{i=j}^r(-1)^{i-j}\zeta^{\star}_{(0,m_1]}(k_i,\ldots,k_{j+1};s)\zeta_{(0,m_2)}(k_{i+1},\ldots,k_r;s)\right)\\
     &=\zeta_{(m_1,m_2)}(k_{1},\ldots,k_r;s)\\
     &\quad+\sum_{i=1}^r(-1)^i\zeta_{(0,m_2)}(k_{i+1},\ldots,k_r;s)\times\left(\sum_{j=1}^i(-1)^{j}\zeta^{\star}_{(0,m_1]}(k_i,\ldots,k_{j+1};s)\zeta_{(0,m_1]}(k_1,\ldots,k_j;s)\right)\\
     &=\zeta_{(m_1,m_2)}(k_{1},\ldots,k_r;s)\qquad(\text{by\ Antipode\ identity})\\
     &\quad+\sum_{i=1}^r(-1)^{i}(-\zeta^{\star}_{(0,m_1]}(k_i,\ldots,k_1;s))\zeta_{(0,m_2)}(k_{i+1},\ldots,k_r;s)\\
     &=\zeta_{(m_1,m_2)}(k_{1},\ldots,k_r;s)-\sum_{i=1}^r(-1)^{i}\zeta^{\star}_{(0,m_1]}(k_i,\ldots,k_1;s)\zeta_{(0,m_2)}(k_{i+1},\ldots,k_r;s),
 \end{align*}
 hence, we have
 \begin{align*}
        \zeta_{(m_1,m_2)}(k_1,\ldots,k_r;s)=\sum_{j=0}^{r} (-1)^j \zeta^\star_{(0,m_1]}(k_{j},\ldots,k_1;s)\cdot\zeta_{(0,m_2)}(k_{j+1},k_{j+2},\ldots,k_r;s).
    \end{align*}
(6) We have
\begin{align*}
    &\zeta_{(m_1,m_2)}(k_{1},\ldots,k_r;s)\\
    &=\sum_{m_1<n_1<\cdots<n_r<m_2}\frac{1}{(n_1+s)^{k_1}\cdots(n_r+s)^{k_r}}\\
    &=\sum_{m_1<n_1<\cdots<n_r<m_2}\frac{1}{n_1^{k_1}\left(1+\frac{s}{n_1}\right)^{k_1}\cdots n_r^{k_r}\left(1+\frac{s}{n_r}\right)^{k_r}}\\
   &=\sum_{m_1<n_1<\cdots<n_r<m_2}\ \sum_{j_1=0}^{\infty}\binom{-k_1}{j_1}\frac{s^{j_1}}{n_1^{k_1+j_1}}\cdots\sum_{j_r=0}^{\infty}\binom{-k_r}{j_r}\frac{s^{j_r}}{n_r^{k_r+j_r}}\\
   &=\sum_{m=0}^{\infty}\left(\sum_{|\bfn|=m}\prod_{l=1}^r\binom{-k_l}{n_l}\zeta_{(m_1,m_2)}(k_1+n_1,\ldots,k_r+n_r)\right)s^m,
\end{align*}
    where $\bfn:=(n_1,\ldots,n_r)\in(\Z_{\geqslant0})^r$ and $|\bfn|:=n_1+\cdots+n_r$, and we used the well-known result
\begin{align*}
(1+x)^{\alpha}=\sum_{n=0}^{\infty} \binom{\alpha}{n}x^n\quad (|x|<1).
\end{align*}
Thus the proposition is proven.
\end{proof}

We obtain that $\zeta_{(m_1,m_2)}(\bfk;s)$ has the following Laurent expansion or Taylor expansion at integer points.

\begin{thm}\label{thm-taylor-expansion-mhzf} Let $\bfk=(k_1,\ldots,k_r)$ be a positive multi-index. Let $m_1,m_2\in\mathbb{Z}\cup\{-\infty,+\infty\},\ m_1<m_2$. If $n\in\mathbb{Z}_{\geqslant-m_1}\cup\mathbb{Z}_{\leqslant-m_2}$, and $|s-n|<1$, then we have
\begin{align*}
\zeta_{(m_1,m_2)}(\bfk;s)=\sum_{m=0}^\infty \left(\sum_{|\bfn|=m} \prod\limits_{l=1}^r \binom{-k_l}{n_l}\zeta_{(n+m_1,n+m_2)}(\bfk+\bfn)\right)(s-n)^m,
\end{align*}
where $\bfn:=(n_1,\ldots,n_r)\in(\Z_{\geqslant0})^r$ and $|\bfn|:=n_1+\cdots+n_r$.
\end{thm}
\begin{proof}
By direct calculations, we obtain
\begin{align*}
\zeta_{(m_1,m_2)}(\bfk;s)&=\zeta_{(n+m_1,n+m_2)}(\bfk;s-n)\qquad(\text{by\ translation})\\
(\text{by\ expansion})&=\sum_{m=0}^\infty\left(\sum_{|\bfn|=m}\prod_{l=1}^r\binom{-k_l}{n_l}\zeta_{(n+m_1,n+m_2)}(\bfk+\bfn)  \right)(s-n)^m.
\end{align*}
This yields the desired result.
\end{proof}

In particular, if $n=0$ and $|s|<1$, then
\begin{align}
\zeta_{(m_1,m_2)}(\bfk;s)=\sum_{m=0}^{\infty}\left( \sum_{|\bfn|=m} \prod\limits_{l=1}^r \binom{-k_l}{n_l} \zeta_{(m_1,m_2)}(\bfk)\right)s^m.
\end{align}

\begin{thm}\label{thm-Laurent-expansion-mhzf} Let $\bfk=(k_1,\ldots,k_r)$ be a positive multi-index. Let $m_1,m_2\in\mathbb{Z}\cup\{-\infty,+\infty\},\ m_1<m_2$. If $-m_2<n<-m_1$, and $|s-n|<1$, then we have
\begin{align}\label{Laurent-expansion-mhzf}
\zeta_{(m_1,m_2)}(\bfk;s)=\left(\sum_{m=0}^{\infty}a_m(s-n)^m+\sum_{j=1}^r\sum_{m=0}^{\infty}b_{m,j}(s-n)^{m-k_j}\right),
\end{align}
where
\begin{align*}
    a_m&:=\sum_{|\bfn|=m}\prod_{l=1}^r\binom{-k_l}{n_l}\sum_{j=0}^r\zeta_{(n+m_1,0)}(\bfk_{[1,j]}+\bfn_{[1,j]})\zeta_{(0,n+m_2)}(\bfk_{(j,r]}+\bfn_{(j,r]}),\\
    b_{m,j}&:=\sum_{|\bfn|-n_j=m}\prod_{l\neq j}\binom{-k_l}{n_l}\zeta_{(n+m_1,0)}(\bfk_{[1,j)}+\bfn_{[1,j)})\zeta_{(0,n+m_2)}(\bfk_{(j,r]}+\bfn_{(j,r]})
\end{align*}
and $\bfn:=(n_1,\ldots,n_r)\in(\Z_{\geqslant0})^r$.
\end{thm}

\begin{proof}
 We have
    \begin{align*}
        &\zeta_{(m_1,m_2)}(\bfk;s)\\
        &=\zeta_{(n+m_1,n+m_2)}(\bfk;s-n)\\
        &=\sum_{j=0}^r\zeta_{(n+m_1,0)}(\bfk_{[1,j]};s-n)\cdot\zeta_{(0,n+m_2)}(\bfk_{(j,r]};s-n)\\
        &\quad+\sum_{j=1}^r\frac{1}{(s-n)^{k_j}}\zeta_{(n+m_1,0)}(\bfk_{[1,j)};s-n)\cdot\zeta_{(0,n+m_2)}(\bfk_{(j,r]};s-n)\\
        &=\sum_{j=0}^r\sum_{n_1=0}^{\infty}\cdots\sum_{n_r=0}^{\infty}\prod_{l=1}^r\binom{-k_l}{n_l}\zeta_{(n+m_1,0)}(\bfk_{[1,j]}+\bfn_{[1,j]})\cdot\zeta_{(0,n+m_2)}(\bfk_{(j,r]}+\bfn_{(j,r]})(s-n)^{|\bfn|}\\
        &+\sum_{j=1}^r\sum_{l\neq j}\sum_{n_l=0}^{\infty}\prod_{i\neq j}\binom{-k_i}{n_i}\zeta_{(n+m_1,0)}(\bfk_{[1,j)}+\bfn_{[1,j)})\cdot\zeta_{(0,n+M)}(\bfk_{(j,r]}+\bfn_{(j,r]})(s-n)^{|\bfn|-n_j-k_j}\\
        &=\sum_{m=0}^{\infty}a_m(s-n)^m+\sum_{j=1}^r\sum_{m=0}^{\infty}b_{m,j}(s-n)^{m-k_j},
\end{align*}
where
\begin{align*}
    a_m&:=\sum_{|\bfn|=m}\prod_{l=1}^r\binom{-k_l}{n_l}\sum_{j=0}^r\zeta_{(n+m_1,0)}(\bfk_{[1,j]}+\bfn_{[1,j]})\zeta_{(0,n+m_2)}(\bfk_{(j,r]}+\bfn_{(j,r]}),\\
    b_{m,j}&:=\sum_{|\bfn|-n_j=m}\prod_{l\neq j}\binom{-k_l}{n_l}\zeta_{(n+m_1,0)}(\bfk_{[1,j)}+\bfn_{[1,j)})\zeta_{(0,n+m_2)}(\bfk_{(j,r]}+\bfn_{(j,r]})
\end{align*}
and $\bfn:=(n_1,\ldots,n_r)\in(\Z_{\geqslant0})^r$.
\end{proof}

\begin{exa}
    In particular, if $r=1,k_1=k$ and $n\in\mathbb{Z}_{<0}$ then ($|s-n|<1$)
\begin{align*}\label{Laurent-expansion-mhzf-cases}
\zeta(k;s)&=\sum_{m=0}^\infty \binom{-k}{m} \zeta_{(n,0)}(k+m)(s-n)^m+\frac1{(s-n)^k}+\sum_{m=0}^\infty \binom{-k}{m} \zeta(k+m)(s-n)^m.
\end{align*}
\end{exa}

\section{Main Ideas}

Flajolet and Salvy \cite{Flajolet-Salvy} defined a kernel function $\xi(s)$ with two requirements: (1). $\xi(s)$ is meromorphic in the whole complex plane. (2). $\xi(s)$ satisfies $\xi(s)=o(s)$ over an infinite collection of circles $\left| s \right| = {\rho _k}$ with ${\rho _k} \to \infty $. Applying these two conditions of kernel
function $\xi(s)$, Flajolet and Salvy discovered the following residue lemma.

\begin{lem}\emph{(cf.\ \cite{Flajolet-Salvy})}\label{lem-redisue-thm}
Let $\xi(s)$ be a kernel function and let $r(s)$ be a rational function which is $O(s^{-2})$ at infinity. Then

\begin{align}
\sum\limits_{\alpha  \in O} {{\mathop{\rm Res}}{{\left( {r(s)\xi(s)},\alpha  \right)}}}  + \sum\limits_{\beta  \in S}  {{\mathop{\rm Res}}{{\left( {r(s)\xi(s)}, \beta  \right)}}}  = 0,
\end{align}
where $S$ is the set of poles of $r(s)$ and $O$ is the set of poles of $\xi(s)$ that are not poles $r(s)$. Here ${\mathop{\rm Re}\nolimits} s{\left( {r(s)},\alpha \right)} $ denotes the residue of $r(s)$ at $s= \alpha.$
\end{lem}

\begin{lem}\emph{(cf.\ \cite{Flajolet-Salvy})}\label{tran-Laext} For any $n\in \Z$,
\begin{align}
&\pi \cot \left( {\pi s} \right)\mathop  = \frac{-2}{s-n}\sum\limits_{k = 0}^\infty  {\zeta( 2k){{\left( {s - n} \right)}^{2k }}},\ \zeta(0)=-1/2,\label{expansion-one-sine}\\
&\frac{\pi }
{{\sin \left( {\pi s} \right)}}\mathop  = \frac{2(-1)^n}{s-n}\sum\limits_{k = 0}^\infty  {\bar \zeta \left( {2k} \right){{\left( {s - n} \right)}^{2k }}},\ \bar{\zeta}(0)=\frac{1}{2}  ,\label{expansion-one-cosine}
\end{align}
where ${\bar \zeta} (s)$ denotes the alternating Riemann zeta function defined by
\[\bar \zeta \left( s \right) := \sum\limits_{n = 1}^\infty  {\frac{{{{\left( { - 1} \right)}^{n - 1}}}}{{{n^s}}}}=(1-2^{1-s})\zeta(s) \quad(\Re(s)>0).\]
\end{lem}

In \cite{Flajolet-Salvy}, Flajolet and Salvy used residue computations on large circular contour and specific functions to obtain more independent relations for Euler sums. These functions are of the form $\xi(s)r(s)$, where $r(s):=1/{s^q}$ and $\xi(s)$ is a product of cotangent (or cosecant) and polygamma functions.

We can try replacing the kernel function in their paper with multiple Hurwitz zeta functions and use contour integration along with the residue theorem to derive some relations among multiple zeta values. For instance, consider the following contour integrals:
\begin{align}\label{contourintegral-one}
\lim_{R\rightarrow \infty}\oint_{C_R} \frac{\pi \cot(\pi s)\zeta_{(m_1,m_2)}(\bfk;s)}{s^q}ds
\end{align}
and
\begin{align}\label{contourintegral-two}
\lim_{R\rightarrow \infty}\oint_{C_R} \frac{\pi \csc(\pi s)\zeta_{(m_1,m_2)}(\bfk;s)}{s^q}ds,
\end{align}
where $C_R$ denote a circular contour with radius $R$. It is straightforward to observe that both of the above contour integrals are equal to zero.

\section{Parity Conjecture and Regularization}
In this section, we first compute the residue of the contour integral \eqref{contourintegral-one} by combining Lemmas \ref{lem-redisue-thm} and \ref{tran-Laext} with Proposition \ref{prop}, Theorems \ref{thm-taylor-expansion-mhzf} and \ref{thm-Laurent-expansion-mhzf} , thereby establishing several formulas involving finite multiple zeta values. We then regularize these formulas and take the limit to derive parity formulas for the regularization of multiple zeta values.

We always assume that $m_1,m_2\in\mathbb{Z}\cup\{-\infty,+\infty\},\ m_1<m_2$, $\bfk=(k_1,\ldots,k_r)$ be a positive multi-index and $q\in\mathbb{Z}_{>1}$.

\begin{thm}\label{parity1}
    If $m_2\leqslant0$ or $m_1\geqslant0$, then we have
    \begin{align*}
        &\sum_{n\leqslant-m_2,n\neq0}\frac{\zeta_{(n+m_1,n+m_2)}(\bfk)}{n^q}+\sum_{n\geqslant-m_1,n\neq0}\frac{\zeta_{(n+m_1,n+m_2)}(\bfk)}{n^q}\\
        &-2\sum_{2k+m=q}\left(\sum_{|\bfn|=m} \prod\limits_{l=1}^r \binom{-k_l}{n_l} \zeta_{(m_1,m_2)}(\bfk+\bfn)\right)\zeta(2k)\\
        &+\sum_{-m_2<n<-m_1}\sum_{j=0}^r\frac{\zeta_{(n+m_1,0)}(\bfk_{[1,j]})\zeta_{(0,n+m_2)}(\bfk_{(j,r]})}{n^q}\\
&-2 \sum_{-m_2<n<-m_1}\sum_{j=1}^r \sum_{2k+m= k_j} \left(\sum_{|\bfn|=m}\binom{-q}{n_j}\prod_{l\neq j}\binom{-k_l}{n_l}\frac{\zeta_{(n+m_1,0)}(\bfk_{[1,j)}+\bfn_{[1,j)})}{n^{q+n_j}}
\atop\times\zeta_{(0,n+m_2)}(\bfk_{(j,r]}+\bfn_{(j,r]})\right)\zeta(2k)\\
&=0,\end{align*}
    where $\bfn=(n_1,\ldots,n_r)\in(\mathbb{Z}_{\geqslant0})^r$.
\end{thm}

\begin{proof}
We consider the following contour integral
\begin{align*}
\lim_{R\rightarrow \infty}\oint_{C_R} F_q(\bfk;s)\de s:=\lim_{R\rightarrow \infty}\oint_{C_R} \frac{\pi \cot(\pi s)\zeta_{(m_1,m_2)}(\bfk;s)}{s^q}\de s=0.
\end{align*}
Clearly, the function $F_q(\bfk;s)$ has only singularities are poles at the integers. It's easy to see that
\begin{align*}
    \frac{1}{s^q}=\sum_{k=0}^{\infty}\binom{-q}{k}\frac{(s-n)^k}{n^{q+k}}\quad (n\in\mathbb{Z}\backslash\{0\}).
\end{align*}
Applying Theorems  \ref{thm-taylor-expansion-mhzf} and \ref{thm-Laurent-expansion-mhzf} along with \eqref{expansion-one-sine}, we can derive the residue values for the following cases through a case-by-case discussion.

\textbf{Case 1}. If $n\leqslant-m_2$ or $n\geqslant-m_1$, then
\begin{itemize}
    \item[(1)] If $n\neq0$, we have
    \begin{align}
{\rm Res}[F_q,s=n]&=\lim_{s\rightarrow n} (s-n) \frac{\pi \cot(\pi s)\zeta_{(m_1,m_2)}(\bfk;s)}{s^q} \nonumber\\
&=\zeta_{(n+m_1,n+m_2)}(\bfk)\cdot\frac{1}{n^q}.
\end{align}

\item[(2)] If $n=0$, we have
\begin{align}
&{\rm Res}[F_q,s=0]=\frac1{q!}\lim_{s\rightarrow 0} \frac{\de^q}{\de s^q}\left\{s^{q+1}\frac{\pi \cot(\pi s)\zeta_{(m_1,m_2)}(\bfk;s)}{s^q}\right\}\nonumber\\
&=-2\sum_{2k+m=q}\left(\sum_{|\bfn|=m} \prod\limits_{l=1}^r \binom{-k_l}{n_l} \zeta_{(m_1,m_2)}(\bfk+\bfn)\right)\zeta(2k).
\end{align}

\textbf{Cases 2}. If $-m_2<n<-m_1$, then we have
\begin{align*}
&{\rm Res}[F_q,s=n]\\
&=\sum_{j=0}^r\frac{\zeta_{(n+m_1,0)}(\bfk_{[1,j]})\zeta_{(0,n+m_2)}(\bfk_{(j,r]})}{n^q}\\
&\quad-2 \sum_{j=1}^r \sum_{2k+m= k_j} \left(\sum_{|\bfn|=m}  \binom{-q}{n_j}\prod_{l\neq j}\binom{-k_l}{n_l}\frac{\zeta_{(n+m_1,0)}(\bfk_{[1,j)}+\bfn_{[1,j)})}{n^{q+n_j}}
\atop\times\zeta_{(0.n+m_2)}(\bfk_{(j,r]}+\bfn_{(j,r]})\right)\zeta(2k).
\end{align*}

\end{itemize}

By Lemma \ref{lem-redisue-thm}, we know that
\begin{align*}
    &\sum_{n\leqslant-m_2,n\neq0}{\rm Res}[F_q,s=n]+\sum_{-m_2<n<-m_1}{\rm Res}[F_q,s=n]\\
    &+{\rm Res}[F_q,s=0]+\sum_{n\geqslant-m_1,n\neq0}{\rm Res}[F_q,s=n]=0.
\end{align*}

Finally, combining these three contributions yields the statement of Theorem \ref{parity1}.
\end{proof}

\begin{thm}\label{parity2}
    If $m_1<0<m_2$, then we have
    \begin{align*}
        &\sum_{n\leqslant-m_2}\frac{\zeta_{(n+m_1,n+m_2)}(\bfk)}{n^q}+\sum_{n\geqslant-m_1}\frac{\zeta_{(n+m_1,n+m_2)}(\bfk)}{n^q}\\
    &-2\sum_{2k+m=q}\left(\sum_{|\bfn|=m}\prod_{l=1}^r\binom{-k_l}{n_l}\sum_{j=0}^r\zeta_{(m_1,0)}(\bfk_{[1,j]}+\bfn_{[1,j]})\zeta_{(0,m_2)}(\bfk_{(j,r]}+\bfn_{(j,r]})\right)\zeta(2k)\\
&-2\sum_{j=1}^r\sum_{2k+m=q+k_j}\left(\sum_{|\bfn|-n_j=m}\prod_{l\neq j}\binom{-k_l}{n_l}\zeta_{(n+m_1,0)}(\bfk_{[1,j)}+\bfn_{[1,j)})\zeta_{(0,n+m_2)}(\bfk_{(j,r]}+\bfn_{(j,r]})\right)\zeta(2k)\\
        &+\sum_{j=0}^r\sum_{-m_2<n<-m_1,n\neq0}\frac{\zeta_{(n+m_1,0)}(\bfk_{[1,j]})\zeta_{(0,n+m_2)}(\bfk_{(j,r]})}{n^q}\\
&-2 \sum_{-m_2<n<-m_1,n\neq0}\sum_{j=1}^r \sum_{2k+m= k_j} \left(\sum_{|\bfn|=m}  \binom{-q}{n_j}\prod_{l\neq j}\binom{-k_l}{n_l}\frac{\zeta_{(n+m_1,0)}(\bfk_{[1,j)}+\bfn_{[1,j)})}{n^{q+n_j}}
\atop\times\zeta_{(0,n+m_2)}(\bfk_{(j,r]}+\bfn_{(j,r]})\right)\zeta(2k)\\
&=0,
    \end{align*}
    where $\bfn=(n_1,\ldots,n_r)\in(\mathbb{Z}_{\geqslant0})^r$ .
\end{thm}

\begin{proof}
We consider the following contour integral
\begin{align*}
\lim_{R\rightarrow \infty}\oint_{C_R} F_q(\bfk;s)\de s:=\lim_{R\rightarrow \infty}\oint_{C_R} \frac{\pi \cot(\pi s)\zeta_{(m_1,m_2)}(\bfk;s)}{s^q}\de s=0.
\end{align*}
Clearly, the function $F_q(\bfk;s)$ has only singularities are poles at the integers.

\textbf{Case 1}. If $n\leqslant-m_2$ or $n\geqslant-m_1$, then we have
\begin{align}
{\rm Res}[F_q,s=n]&=\lim_{s\rightarrow n} (s-n) \frac{\pi \cot(\pi s)\zeta_{(m_1,m_2)}(\bfk;s)}{s^q} \nonumber\\
&=\zeta_{(n+m_1,n+m_2)}(\bfk)\cdot\frac{1}{n^q}.
\end{align}

\textbf{Cases 2}. If $-m_2<n<-m_1$, then
\begin{itemize}
    \item[(1)] If $n\neq0$, we have
    \begin{align*}
&{\rm Res}[F_q,s=n]\\
&=\sum_{j=0}^r\frac{\zeta_{(n+m_1,0)}(\bfk_{[1,j]})\zeta_{(0,n+m_2)}(\bfk_{(j,r]})}{n^q}\\
&\quad-2 \sum_{j=1}^r \sum_{2k+m= k_j} \left(\sum_{|\bfn|=m}  \binom{-q}{n_j}\prod_{l\neq j}\binom{-k_l}{n_l}\frac{\zeta_{(n+m_1,0)}(\bfk_{[1,j)}+\bfn_{[1,j)})}{n^{q+n_j}}
\atop\times\zeta_{(0,n+m_2)}(\bfk_{(j,r]}+\bfn_{(j,r]})\right)\zeta(2k).
\end{align*}

\item[(2)] If $n=0$, we have
\begin{align*}
&{\rm Res}[F_q,s=0]\\
&=-2\sum_{2k+m=q}\left(\sum_{|\bfn|=m}\prod_{l=1}^r\binom{-k_l}{n_l}\sum_{j=0}^r\zeta_{(m_1,0)}(\bfk_{[1,j]}+\bfn_{[1,j]})\zeta_{(0,m_2)}(\bfk_{(j,r]}+\bfn_{(j,r]})\right)\zeta(2k)\\
&-2\sum_{j=1}^r\sum_{2k+m=q+k_j}\left(\sum_{|\bfn|-n_j=m}\prod_{l\neq j}\binom{-k_l}{n_l}\zeta_{(n+m_1,0)}(\bfk_{[1,j)}+\bfn_{[1,j)})\zeta_{(0,n+m_2)}(\bfk_{(j,r]}+\bfn_{(j,r]})\right)\zeta(2k).
\end{align*}

\end{itemize}

By Lemma \ref{lem-redisue-thm}, we know that
\begin{align*}
    &\sum_{n\leqslant-m_2}{\rm Res}[F_q,s=n]+\sum_{-m_2<n<-m_1,n\neq0}{\rm Res}[F_q,s=n]\\
    &+{\rm Res}[F_q,s=0]+\sum_{n\geqslant-m_1}{\rm Res}[F_q,s=n]=0.
\end{align*}

Finally, combining these two cases yields the statement of Theorem \ref{parity2}.
\end{proof}

\begin{cor}

\label{thm-residue-contour-one} Let $M\gg1$be a fixed integer, for $q\in\Z_{>1}$ and $\bfk=(k_1,\ldots,k_r)$ be a positive multi-index, we have
\begin{align*}
    &\sum_{j=0}^r(-1)^j\sum_{n=1}^{+\infty}\frac{\zeta^{\star}_{(0,n]}(k_j,\ldots,k_1)\zeta_{(0,n+M)}(k_{j+1},\cdots,k_r)}{n^q}\\
    &-2\sum_{2k+m=q}\left(\sum_{|\bfn|=m} \prod\limits_{l=1}^r \binom{-k_l}{n_l} \zeta_{(0,M)}(k_1+n_1,\ldots,k_r+n_r)\right)\zeta(2k)\\
    &+\sum_{-M<n<0}\sum_{j=0}^r\frac{\zeta_{(n,0)}(k_1,\ldots,k_j)}{n^q}\zeta_{(0,n+M)}(k_{j+1},k_{j+2},\ldots,k_r)\\
&\quad-2\sum_{-M<n<0} \sum_{j=1}^r \sum_{2k+m= k_j} \left(\sum_{|\bfn|=m}  \binom{-q}{n_j}\prod_{l\neq j}\binom{-k_l}{n_l}\frac{\zeta_{(n,0)}(k_1+n_1,\ldots,k_{j-1}+n_{j-1})}{n^{q+n_j}}
\atop\times\zeta_{(0,n+M)}(k_{j+1}+n_{j+1},\ldots,k_r+n_r)\right)\zeta(2k)\\
&+\sum_{n\leqslant-M}\frac{\zeta_{(n,n+M)}(k_1,\ldots,k_r)}{n^q}=0,
\end{align*}
where $\bfn=(n_1,\ldots,n_r)\in(\mathbb{Z}_{\geqslant0})^r$ .
\end{cor}

\begin{proof}
    Let $m_1=0,m_2=M$, by using Theorem \ref{parity1} and truncation formula in Proposition \ref{prop}
    \begin{align*}
        \zeta_{(m_1,m_2)}(k_1,\ldots,k_r;s)=\sum_{j=0}^{r} (-1)^j \zeta^\star_{(0,m_1]}(k_{j},\ldots,k_1;s) \zeta_{(0,m_2)}(k_{j+1},k_{j+2},\ldots,k_r;s),
    \end{align*}
    we obtain the desired conclusion.
\end{proof}

In order to provide the stuffle regularization, we first need to establish some preliminary lemmas.

\begin{lem}\label{lemma-1}
    Let $\bfk=(k_1,\ldots,k_r)\in(\mathbb{Z}_{>0})^r$ be a positive multi-index, then for all $n\in\mathbb{Z}_{>0}$, we have
    $$\zeta^{\star}_{(0,n]}(\bfk)<2^{r-1}(1+\ln(n))^{r}.$$
\end{lem}
\begin{proof} By direct calculations, we have
    \begin{align*}
        \zeta^{\star}_{(0,n]}(\bfk)&\leqslant\zeta^{\star}_{(0,n]}(\{1\}_r)\\
        &=\sum_{0<n_1\leqslant\cdots\leqslant n_r\leqslant n}\frac{1}{n_1\cdots n_r}\\
        &=\sum_{0<n_1<\cdots<n_r\leqslant n}\frac{1}{n_1\cdots n_r}\\
        &\quad+\sum_{0<n_1\leqslant<n_2<\cdots< n_r\leqslant n}\frac{1}{n_1\cdots n_r}+\cdots+\sum_{0<n_1<\cdots<n_{r-1}\leqslant n_r\leqslant n}\frac{1}{n_1\cdots n_r}\\
        &\quad+\sum_{0<n_1=\cdots=n_r\leqslant n}\frac{1}{n_1\cdots n_r}\\
        &<\sum_{j=0}^{r-1}\binom{r-1}{r-1-j}\zeta_{(0,n]}(\{1\}_{j+1})\\
        &<\sum_{j=0}^{r-1}\binom{r-1}{r-1-j}\frac{(\zeta_{(0,n]}(1))^{j+1}}{(j+1)!}\quad(\text{since}\ y_1^{*{j+1}}=(j+1)!\cdot y_1^j+\cdots)\\
        &<\sum_{j=0}^{r-1}\binom{r-1}{r-1-j}(\zeta_{(0,n]}(1))^{r}\\
        &<2^{r-1}(1+\ln(n))^{r}.
    \end{align*}
This completes the proof of the lemma.
\end{proof}

\begin{lem}\label{lem-admisslble}
    Let $\bfk=(k_1,\ldots,k_r)$ be an admissible multi-index . Let $M>>1$ be a fixed integer and $n\in\mathbb{Z}$, if $M+n>0$ then we have
\begin{align*}
    \left|\zeta_{(0,n+M)}(k_{1},\ldots,k_r)-\zeta(k_{1},\ldots,k_r)\right|<2r\frac{(1+\ln(n+M))^{r-1}}{n+M}.
\end{align*}
\end{lem}
\begin{proof}
By elementary calculations, we have
    \begin{align*}
    &\left|\zeta(k_{1},\ldots,k_r)-\zeta_{(0,n+M)}(k_{1},\ldots,k_r)\right|\\
        &=\left|\zeta_{(0,+\infty)}(k_{1},\ldots,k_r)-\zeta_{(0,n+M)}(k_{1},\ldots,k_r)\right|\\
        &=\left|\sum_{j=0}^r\zeta_{(0,n+M)}(k_{1},\ldots,k_j)\zeta_{[n+M,+\infty)}(k_{j+1},\ldots,k_r)-\zeta_{(n,n+M)}(k_{1},\ldots,k_r)\right|\\
        &=\sum_{j=0}^{r-1}\zeta_{(0,n+M)}(k_{1},\ldots,k_j)\zeta_{[n+M,+\infty)}(k_{j+1},\ldots,k_r)\\
        &\leqslant\sum_{j=0}^{r-1}\zeta_{(0,n+M)}(\{1\}_{j})\zeta_{[n+M,+\infty)}(\{1\}_{r-j-1},2)\\
        &<\sum_{j=0}^{r-1}(\zeta_{(0,n+M)}(1))^{j}\zeta_{[n+M,+\infty)}(\{1\}_{r-j-1},2)\\
        &< \sum_{j=0}^{r-1}(1+\ln(n+M))^j\cdot\frac{2}{n+M}\\
        &<2r\frac{(1+\ln(n+M))^{r-1}}{n+M}.
    \end{align*}
Thus, the desired result is obtained.
\end{proof}

\begin{lem}\label{lem-positive}
     Let $\bfk=(k_j,\ldots,k_1)$ be a positive multi-index, $s\in\mathbb{Z}_{\geqslant0}$ and $q\in\mathbb{Z}_{>1}$. Let $M>>1$ be a fixed integer, then we have
\begin{align*}
    &\left|\sum_{n=1}^{+\infty}\frac{\zeta^{\star}_{(0,n]}(k_j,\ldots,k_1)(\zeta_{(0,n+M)}(1))^s}{n^q}-\sum_{n=1}^{+\infty}\frac{\zeta^{\star}_{(0,n]}(k_j,\ldots,k_1)(\zeta_{(0,M)}(1))^s}{n^q}\right|\\
    &<c\cdot\frac{\ln^{j+s}(M)}{M},
\end{align*}
where the constant $c$ is independent on $M$.
\end{lem}

\begin{proof} Firstly, we have
    \begin{align*}
        |(\zeta_{(0,n+M)}(1))^s-(\zeta_{(0,M)}(1))^s|&=(\zeta_{(0,M)}(1)+\zeta_{[M,n+M)}(1))^s-(\zeta_{(0,M)}(1))^s\\
        &=\sum_{l=1}^{s}\binom{s}{l}(\zeta_{(0,M)}(1))^{s-l}(\zeta_{[M,n+M)}(1))^l\\
        &<\sum_{l=1}^{s}\binom{s}{l}(1+\ln(M))^{s-l}\ln^l\left(1+\frac{n}{M-1}\right)\\
        &<\sum_{l=1}^{s}\binom{s}{l}2^{s-l}\ln^{s-l}(M)\ln^l\left(1+\frac{n}{M-1}\right)
    \end{align*}
    by Lemma \ref{lemma-1}, we know that $\zeta^{\star}_{(0,n]}(k_j,\ldots,k_1)<2^{j-1}(1+\ln(n))^j$. Therefore, the sum to be estimated can be expressed as a linear combination of
$$\sum_{n=1}^{+\infty}\frac{\ln^j(n)\ln^l\left(1+\frac{n}{M-1}\right)}{n^q},$$
notice that $q>1$, we obtain
\begin{align*}
&\sum_{n=1}^{+\infty}\frac{\ln^j(n)\ln^l\left(1+\frac{n}{M-1}\right)}{n^q}\\
    &<\sum_{n=1}^{+\infty}\frac{\ln^j(n)\ln^l\left(1+\frac{n}{M-1}\right)}{n^2}\\
    &=\sum_{n=1}^{M}\frac{\ln^j(n)\ln^l\left(1+\frac{n}{M-1}\right)}{n^2}+\sum_{n=M+1}^{+\infty}\frac{\ln^j(n)\ln^l\left(1+\frac{n}{M-1}\right)}{n^2}\\
    &<\frac{1}{(M-1)^l}\sum_{n=1}^Mn^{l-2}\ln^j(n)+\sum_{n=M+1}^{+\infty}\frac{\ln^j(n)\ln^l(1+n)}{n^2}\\
    &<c_1\frac{\ln^j(M)}{M}+2^l\sum_{n=M+1}^{+\infty}\frac{\ln^{j+l}(n)}{n^2}\\
    &<c_1\frac{\ln^j(M)}{M}+2^l\int_{M}^{+\infty}\frac{\ln^{j+l}(x)}{x^2}dx\\
    &<c_1\frac{\ln^j(M)}{M}+2^lc_2\frac{\ln^{j+l}(M)}{M}\\
    &<c\frac{\ln^{j+l}(M)}{M}.
\end{align*}
Thus, the proof of the lemma is completed.
\end{proof}

\begin{thm}\label{thm-star}
 Let $\bfk$ be a positive multi-index and $q\in\mathbb{Z}_{>1}$. Let $M$ be a sufficiently large positive real number, then we have
 \begin{align*}
    &\left|\sum_{n=1}^{+\infty}\frac{\zeta^{\star}_{(0,n]}(k_j,\ldots,k_1)\zeta_{(0,n+M)}(k_{j+1},\ldots,k_r)}{n^q}-\sum_{n=1}^{+\infty}\frac{\zeta^{\star}_{(0,n]}(k_j,\ldots,k_1)\zeta_*^{T=\zeta_{(0,M)}(1)}(k_{j+1},\ldots,k_r)}{n^q}\right|\\
    &<c\frac{\ln^{r}(M)}{M},
\end{align*}
where the constant $c$ is independent on $M$.
\end{thm}

\begin{proof}
It is widely known that $(\Q\langle Y\rangle,*)\cong(\Q\langle Y\rangle^0,*)[y_1]$ (For a more detailed description, the reader is referred to Section 7 of Chapter 1 in \cite{BJ}). Hence, we have

    $$y_{k_{j+1}}y_{k_{j+2}}\cdots y_{k_r}=w_0*y_1^{*s}+w_1*y_1^{*(s-1)}+\cdots+w_{s-1}*y_1+w_s,$$
    where $w_s,w_{s-1},\cdots,w_1,w_0$ are all admissible words. We obtain
    \begin{align*}
        \zeta_{(0,n+M)}(k_{j+1},k_{j+2},\ldots,k_r)&=\zeta_{(0,n+M)}(w_0)\cdot(\zeta_{(0,n+M)}(1))^s+\zeta_{(0,n+M)}(w_1)\cdot(\zeta_{(0,n+M)}(1))^{s-1}\\
        &\quad+\cdots+\zeta_{(0,n+M)}(w_s)
    \end{align*}
    therefore, we only need to consider $\zeta_{(0,n+M)}(w_i)\cdot(\zeta_{(0,n+M)}(1))^{s-i},\ i=0,\cdots,s$. By Lemmas \ref{lem-admisslble} and \ref{lem-positive} we have
    \begin{align*}
    &\left|\sum_{n=1}^{+\infty}\frac{\zeta^{\star}_{(0,n]}(k_j,\ldots,k_1)\zeta_{(0,n+M)}(w_i)(\zeta_{(0,n+M)}(1))^{s-i})}{n^q}-\sum_{n=1}^{+\infty}\frac{\zeta^{\star}_{(0,n]}(k_j,\ldots,k_1)\zeta(w_i)(\zeta_{(0,M)}(1))^{s-i}}{n^q}\right|\\
    &\leqslant\left|\sum_{n=1}^{+\infty}\frac{\zeta^{\star}_{(0,n]}(k_j,\ldots,k_1)\zeta_{(0,n+M)}(w_i)(\zeta_{(0,n+M)}(1))^{s-i})}{n^q}-\sum_{n=1}^{+\infty}\frac{\zeta^{\star}_{(0,n]}(k_j,\ldots,k_1)\zeta_{(0,n+M)}(w_i)(\zeta_{(0,M)}(1))^{s-i}}{n^q}\right|\\
    &\quad+\left|\sum_{n=1}^{+\infty}\frac{\zeta^{\star}_{(0,n]}(k_j,\ldots,k_1)\zeta_{(0,n+M)}(w_i)(\zeta_{(0,M)}(1))^{s-i})}{n^q}-\sum_{n=1}^{+\infty}\frac{\zeta^{\star}_{(0,n]}(k_j,\ldots,k_1)\zeta(w_i)(\zeta_{(0,M)}(1))^{s-i}}{n^q}\right|\\
    &<\zeta(w_i)\left|\sum_{n=1}^{+\infty}\frac{\zeta^{\star}_{(0,n]}(k_j,\ldots,k_1)(\zeta_{(0,n+M)}(1))^{s-i})}{n^q}-\sum_{n=1}^{+\infty}\frac{\zeta^{\star}_{(0,n]}(k_j,\ldots,k_1)(\zeta_{(0,M)}(1))^{s-i}}{n^q}\right|\\
    &\quad+(\zeta_{(0,M)}(1))^{s-i}\sum_{n=1}^{+\infty}\frac{\zeta^{\star}_{(0,n]}(k_j,\ldots,k_1)|\zeta_{(0,n+M)}(w_i)-\zeta(w_i)|}{n^q}\\
    &<c_1\cdot\zeta(w_i)\frac{\ln^{j+s-i}(M)}{M}+
    c_i\frac{\ln^{r-j-1}(M)}{M}\\
    &<c\frac{\ln^r(M)}{M},
    \end{align*}
    hence, we obtain the desired conclusion.
\end{proof}

\begin{lem}\label{lem-negative-1}
    Let $\bfk=(k_1,\ldots,k_r)(k_r>1)$ be an admissible multi-index and $q\in\mathbb{Z}_{>1}$. Let $M>>1$ be a fixed integer, then we have
    $$\left|\sum_{-M<n<0}\frac{\zeta_{(n,0)}(k_1,\ldots,k_j)}{n^q}\zeta_{(0,n+M)}(k_{j+1},\ldots,k_r)-\sum_{n<0}\frac{\zeta_{(n,0)}(k_1,\ldots,k_j)}{n^q}\zeta(k_{j+1},\ldots,k_r)\right|<c\frac{\ln^{r-1}(M)}{M},$$
    where $j=0,1,\cdots,r$ and the constant $c$ is independent on $M$.
\end{lem}

\begin{proof}If $j=r$, this is obvious. For $j=0,1,\cdots,r-1$, since
    \begin{align*}
        &\sum_{-M<n<0}\frac{\zeta_{(n,0)}(k_1,\ldots,k_j)}{n^q}\zeta_{(0,n+M)}(k_{j+1},\ldots,k_r)\\
        &=(-1)^{q+k_1+\cdots+k_j}\sum_{0<n<M}\frac{\zeta_{(0,n)}(k_j,\ldots,k_1)\cdot\zeta_{(0,M-n)}(k_{j+1},\ldots,k_r)}{n^q},
    \end{align*}
    hence
    \begin{align*}
        &\left|\sum_{0<n<M}\frac{\zeta_{(0,n)}(k_j,\ldots,k_1)\cdot\zeta_{(0,M-n)}(k_{j+1},\ldots,k_r)}{n^q}-\sum_{0<n<M}\frac{\zeta_{(0,n)}(k_j,\ldots,k_1)\cdot\zeta(k_{j+1},\ldots,k_r)}{n^q}\right|\\
        &<2r\sum_{0<n<M}\frac{\zeta_{(0,n)}(k_j,\ldots,k_1)}{n^q}\cdot\frac{(1+\ln(M-n))^{r-j-1}}{M-n}\qquad(\text{by\ Lemma\ \ref{lem-admisslble}})\\
        &<r2^r\sum_{0<n<M}\frac{(1+\ln(n)))^j}{n^2}\cdot\frac{(1+\ln(M-n))^{r-j-1}}{M-n}\qquad(\text{by\ Lemma\ \ref{lemma-1}})\\
        &=r2^r\left(\frac{(1+\ln(M-1))^{r-j-1}}{M-1}+\frac{(1+\ln(2))^j}{2^2}\cdot\frac{(1+\ln(M-2))^{r-j-1}}{M-2}\right)\\
        &\quad+r2^r\sum_{2<n<M-2}\frac{(1+\ln(n)))^j}{n^2}\cdot\frac{(1+\ln(M-n))^{r-j-1}}{M-n}\\
        &\quad+r2^r\left(\frac{(1+\ln(M-2))^j}{(M-2)^2}\cdot\frac{(1+\ln(2))^{r-j-1}}{2}+\frac{(1+\ln(M-1))^{j}}{(M-1)^2}\right)\\
        &<c_1\frac{\ln^{r-1}(M)}{M}+r4^r\sum_{2<n<M-2}\frac{\ln^j(n)\ln^{r-j-1}(M-n)}{n^2(M-n)}\\
        &=c_1\frac{\ln^{r-1}(M)}{M}+r4^r\sum_{2<n\leqslant\left\lfloor M/2\right\rfloor}\frac{\ln^j(n)\ln^{r-j-1}(M-n)}{n^2(M-n)}+r4^r\sum_{\lfloor M/2\rfloor<n<M-2}\frac{\ln^j(n)\ln^{r-j-1}(M-n)}{n^2(M-n)}\\
        &<c_1\frac{\ln^{r-1}(M)}{M}+c_2\frac{\ln^{r-j-1}(M)}{M}\sum_{2<n\leqslant\lfloor M/2\rfloor}\frac{\ln^j(n)}{n^2}+c_3\sum_{\lfloor M/2\rfloor<n<M-2}\frac{\ln^j(n)}{n^2}\\
        &<c_1\frac{\ln^{r-1}(M)}{M}+c_2\left(\sum_{n=1}^{+\infty}\frac{\ln^j(n)}{n^2}\right)\frac{\ln^{r-j-1}(M)}{M}+c_3\int_{M/2}^M\frac{\ln^j(x)}{x^2}\de x\\
        &<c'\frac{\ln^{r-1}(M)}{M},
    \end{align*}
     finally, its easy to see that
     $$\left|\sum_{n\leqslant-M}\frac{\zeta_{(n,0)}(k_1,\ldots,k_j)}{n^q}\zeta(k_{j+1},\ldots,k_r)\right|<c''\frac{\ln^{r-1}(M)}{M}.$$
     The proof is complete.
\end{proof}

\begin{lem}\label{lem-negative-2}
    Let $\bfk=(k_1,\ldots,k_r)$ be a positive multi-index and $q\in\mathbb{Z}_{>1}$. Let $M>>1$ be a fixed integer, then we have
    $$\left|\sum_{-M<n<0}\frac{\zeta_{(n,0)}(\bfk)}{n^q}(\zeta_{(0,n+M)}(1))^s-\sum_{n<0}\frac{\zeta_{(n,0)}(\bfk)}{n^q}(\zeta_{(0,M)}(1))^s\right|<c\frac{\ln^{r+s}(M)}{\sqrt{M}}.$$
    where the constant $c$ is independent on $M$.
\end{lem}

\begin{proof}
Since
\begin{align*}
    (\zeta_{(0,M+n)}(1))^s-(\zeta_{(0,M)}(1))^s=\sum_{i=1}^{s}(-1)^i\binom{s}{i}(\zeta_{(0,M)}(1))^{s-i}(\zeta_{[M+n,M)}(1))^i,
\end{align*}
and
\begin{align*}
    &\left|\sum_{-M<n<0}\frac{\zeta_{(n,0)}(\bfk)}{n^q}(\zeta_{[M+n,M)}(1))^i\right|\\
    &\leqslant\sum_{0<n<M}\frac{\zeta_{(0,n)}(\{1\}_r)}{n^2}(\zeta_{[M-n,M)}(1))^i\\
    &<\sum_{0<n<M}\frac{(\zeta_{(0,n)}(1))^r}{n^2}(\zeta_{[M-n,M)}(1))^i\\
    &<c_1\frac{\ln^{r+i}(M)}{M}+c_2\sum_{2<n<M-2}\frac{\ln^r(n)}{n^2}\ln^i\left(\frac{M-1}{M-1-n}\right)\\
    &<c_1\frac{\ln^{r+i}(M)}{M}+c_2\sum_{2<n\leqslant\sqrt{M}}\frac{\ln^r(n)}{n^2}\ln^i\left(\frac{M-1}{M-1-n}\right)+c_2\sum_{\sqrt{M}<n<M-2}\frac{\ln^r(n)}{n^2}\ln^i\left(\frac{M-1}{M-1-n}\right)\\
    &<c_1\frac{\ln^{r+i}(M)}{M}+c_2\ln^i\left(\frac{M-1}{M-1-\sqrt{M}}\right)\sum_{2<n\leqslant\sqrt{M}}\frac{\ln^r(n)}{n^2}+c_2\ln^i\left(M-1\right)\sum_{\sqrt{M}<n<M-2}\frac{\ln^r(n)}{n^2}\\
    &<c_1\frac{\ln^{r+i}(M)}{M}+c_2'\ln^i\left(\frac{M-1}{M-1-\sqrt{M}}\right)+c_2\ln^i\left(M-1\right)\int_{\sqrt{M}}^{+\infty}\frac{\ln^r(x)}{x^2}dx\\
    &<c_1\frac{\ln^{r+i}(M)}{M}+c_2'\ln^i\left(\frac{M-1}{M-1-\sqrt{M}}\right)+c_2''\frac{\ln^{i+r}\left(M\right)}{\sqrt{M}}\\
    &<c_3\frac{\ln^{r+s}(M)}{\sqrt{M}}.
    \end{align*}
Finally, its easy to see that
     $$\left|\sum_{n\leqslant-M}\frac{\zeta_{(n,0)}(\bfk)}{n^q}(\zeta_{(0,M)}(1))^s\right|<c_3'\frac{\ln^{r+s}(M)}{M}.$$
      The proof is complete.
\end{proof}

\begin{thm}\label{thm-M}
 Let $\bfk$ be a positive multi-index, $q\in\mathbb{Z}_{>1}$. Let $M$ be a sufficiently large positive real number, then we have
 \begin{align*}
    &\left|\sum_{-M<n<0}\frac{\zeta_{(n,0)}(k_1,\ldots,k_j)\zeta_{(0,n+M)}(\bfk)}{n^q}-\sum_{n<0}\frac{\zeta_{(n,0)}(k_1,\ldots,k_j)\zeta_*^{T=\zeta_{(0,M)}(1)}(\bfk)}{n^q}\right|\\
    &<c\frac{\ln^{r+s}(M)}{\sqrt{M}},
\end{align*}
where the constant $c$ is independent on $M$.
\end{thm}
\begin{proof}
    Let $\bfk=(k_1,\ldots,k_l)$ be a positive multi-index, similar to the argument in Theorem \ref{thm-star}, we also have
    $$y_{k_1}y_{k_2}\cdots y_{k_r}=w_0*y_1^{*s}+w_1*y_1^{*(s-1)}+\cdots+w_{s-1}*y_1+w_s,$$
    where $w_s,w_{s-1},\cdots,w_1,w_0$ are all admissible words. By Lemmas \ref{lem-negative-1} and \ref{lem-negative-2} we have
    \begin{align*}
    &\left|\sum_{0<n<M}\frac{\zeta_{(0,n)}(k_j,\ldots,k_1)\zeta_{(0,M-n)}(w_i)(\zeta_{(0,M-n)}(1))^{s-i})}{n^q}-\sum_{0<n<M}\frac{\zeta_{(0,n)}(k_j,\ldots,k_1)\zeta(w_i)(\zeta_{(0,M)}(1))^{s-i}}{n^q}\right|\\
    &\leqslant\left|\sum_{0<n<M}\frac{\zeta_{(0,n)}(k_j,\ldots,k_1)\zeta_{(0,M-n)}(w_i)(\zeta_{(0,M-n)}(1))^{s-i}}{n^q}-\sum_{0<n<M}\frac{\zeta_{(0,n)}(k_j,\ldots,k_1)\zeta(w_i)(\zeta_{(0,M-n)}(1))^{s-i}}{n^q}\right|\\
    &\quad+\left|\sum_{0<n<M}\frac{\zeta_{(0,n)}(k_j,\ldots,k_1)\zeta(w_i)(\zeta_{(0,M-n)}(1))^{s-i})}{n^q}-\sum_{0<n<M}\frac{\zeta_{(0,n)}(k_j,\ldots,k_1)\zeta(w_i)(\zeta_{(0,M)}(1))^{s-i}}{n^q}\right|\\
    &\leqslant(\zeta_{(0,M)}(1))^{s-i}\left|\sum_{0<n<M}\frac{\zeta_{(0,n)}(k_j,\ldots,k_1)\zeta_{(0,M-n)}(w_i)}{n^q}-\sum_{0<n<M}\frac{\zeta_{(0,n)}(k_j,\ldots,k_1)\zeta(w_i)}{n^q}\right|\\
    &\quad+\left|\sum_{0<n<M}\frac{\zeta_{(0,n)}(k_j,\ldots,k_1)\zeta(w_i)(\zeta_{(0,M-n)}(1))^{s-i})}{n^q}-\sum_{0<n<M}\frac{\zeta_{(0,n)}(k_j,\ldots,k_1)\zeta(w_i)(\zeta_{(0,M)}(1))^{s-i}}{n^q}\right|\\
    &<c_1\frac{\ln^{s-i+r}(M)}{M}+c_2\frac{\ln^{r+s-i}(M)}{\sqrt{M}}\\
    &<c\frac{\ln^{r+s}(M)}{\sqrt{M}}.
    \end{align*}
    Hence, we obtain the desired conclusion.
\end{proof}

\begin{lem}\label{lem-M}
    Let $\bfk=(k_1,\ldots,k_r)\ (k_r>1)$ be a positive multi-index and $k\in\mathbb{Z}_{>1}$. Let $M>>1$ be a fixed integer, then we have
    $$\left|\sum_{n\leqslant-M}\frac{\zeta_{(n,n+M)}(k_1,\ldots,k_r)}{n^k}\right|<2^r(r+2)\cdot\frac{\ln^r(M)}{M}.$$
\end{lem}
\begin{proof} Through direct calculation, it follows that
    \begin{align*}
        \left|\sum_{n\leqslant-M}\frac{\zeta_{(n,n+M)}(k_1,\ldots,k_r)}{n^k}\right|&=\left|(-1)^{q+|\bfk|}\sum_{n\geqslant M}\frac{\zeta_{(n-M,n)}(k_r,\ldots,k_1)}{n^k}\right|\ (\text{by reflection})\\
        &=\sum_{n=0}^{\infty}\frac{\zeta_{(0,n+M)}(k_r,\ldots,k_1)}{(n+M)^k}\\
        &\leqslant\sum_{n=0}^{\infty}\frac{\zeta_{(0,n+M)}(\{1\}_r)}{(n+M)^2}\\
        &<\frac{1}{r!}\sum_{n=0}^{\infty}\frac{(\zeta_{(0,n+M)}(1))^r}{(n+M)^2}\\
        &<\frac{1}{r!}\sum_{n=0}^{\infty}\frac{\left(1+\int_1^{n+M}\frac{1}{x}dx\right)^r}{(n+M)^2}\\
        &<\frac{2^r}{r!}\sum_{n=0}^{\infty}\frac{\ln^r(n+M)}{(n+M)^2}\\
        &<\frac{2^r}{r!}\sum_{n=M}^{\infty}\frac{\ln^r(n)}{n^2}\\
        &<2^r(r+2)\cdot\frac{\ln^r(M)}{M}.
    \end{align*}
Therefore, we have completed the proof of this lemma.
\end{proof}

Based on the results of the aforementioned theorems and lemmas, we can now present the parity formulas for the double regularization of multiple zeta values.
\begin{thm}\label{Stuffle-Regularization}(Stuffle Regularization)
For $q\in\Z_{>1}$ and $\bfk=(k_1,\ldots,k_r)\in (\Z_{>0})^r$, we have
    \begin{align*}
&\sum_{j=0}^r (-1)^j \zeta^\star(k_j,\ldots,k_1,q)\zeta_*^T(k_{j+1},\ldots,k_r)\\
&-2\sum_{2k+m=q}\left(\sum_{|\bfn|=m} \prod\limits_{l=1}^r \binom{-k_l}{n_l} \zeta_*^T(k_1+n_1,\ldots,k_r+n_r)\right)\zeta(2k)\\
&+ \sum_{j=0}^r (-1)^{q+k_1+\cdots+k_j} \zeta(k_j,\ldots,k_1,q)\zeta_*^T(k_{j+1},\ldots,k_r)\\
&-2 \sum_{j=1}^r(-1)^{q+|\bfk_{[1,j)}|}\sum_{2k+m= k_j} \left(\sum_{|\bfn|=m} \binom{-q}{n_j}\prod_{l\neq j}\binom{-k_l}{n_l}\zeta_*^T(k_{j+1}+n_{j+1},\ldots,k_r+n_r)
\atop\times(-1)^{|\bfn_{[1,j]}|}\zeta(k_{j-1}+n_{j-1},\ldots,k_1+n_1,q+n_j)\right)\zeta(2k)=0,
\end{align*}
where $\bfn=(n_1,\ldots,n_r)\in(\mathbb{Z}_{\geqslant0})^r$ and $\zeta_*^T(\bfk)$ denote the stuffle regularization.
\end{thm}
\begin{proof}Let $T=\zeta_{(0,M)}(1)$, according to the Corollary \ref{thm-residue-contour-one}, we only need to consider the following four parts.

   \textbf{Part 1}. It follows from Theorem \ref{thm-star} that, as $M\to+\infty$
\begin{align*}
    &\sum_{n=1}^{+\infty}\frac{\zeta^{\star}_{(0,n]}(k_j,\ldots,k_1)\zeta_{(0,n+M)}(k_{j+1},\ldots,k_r)}{n^q}\\
    &=\zeta^{\star}(k_j,\ldots,k_1,q)\zeta_*^T(k_{j+1},\ldots,k_r)+O\left(\frac{\ln^{t_1}(M)}{M}\right)
\end{align*}
for some $t_1\in\mathbb{Z}_{>0}$.

\textbf{Part 2}. According to the stuffle regularization of multiple zeta values, we have
$$\zeta_{(0,M)}(k_1+n_1,\ldots,k_r+n_r)=\zeta_*^T(k_1+n_1,\ldots,k_r+n_r).$$

\textbf{Part 3}. It follows from Theorem \ref{thm-M} that, as $M\to+\infty$

\begin{align*}
    &\sum_{-M<n<0}\frac{\zeta_{(n,0)}(k_1+n_1,\ldots,k_{j-1}+n_{j-1})}{n^{q+n_j}}\cdot \zeta_{(0,n+M)}(k_{j+1}+n_{j+1},\ldots,k_r+n_r)\\
    &=(-1)^{q+|k_{[1,j)]}|+|n_{[1,j]}|}\zeta(k_{j-1}+n_{j-1},\ldots,k_1+n_1,q+n_j)\cdot \zeta_*^T(k_{j+1}+n_{j+1},\ldots,k_r+n_r)\\
    &\quad+O\left(\frac{\ln^{t_2}(M)}{\sqrt{M}}\right)
\end{align*}
for some $t_2\in\mathbb{Z}_{>0}$.

\textbf{Part 4}. Finally, according to Lemma \ref{lem-M}, as $M\to+\infty$, we have
$$\sum_{n\leqslant-M}\frac{\zeta_{(n,n+M)}(k_1,\ldots,k_r)}{n^q}=O\left(\frac{\ln^{t_3}(M)}{M}\right)$$
   for some $t_3\in\mathbb{Z}_{>0}$.

   Combining the four parts above, we obtain the desired conclusion.
\end{proof}
Similarly, we also have shuffle regularization.

\begin{thm}\label{Shuffle-Regularization}(Shuffle Regularization)
For $q\in\Z_{>1}$ and $\bfk=(k_1,\ldots,k_r)\in (\Z_{>0})^r$, we have
    \begin{align*}
&\sum_{j=0}^r (-1)^j \zeta^\star(k_j,\ldots,k_1,q)\zeta_{\shuffle}^T(k_{j+1},\ldots,k_r)\\
&-2\sum_{2k+m=q}\left(\sum_{|\bfn|=m} \prod\limits_{l=1}^r \binom{-k_l}{n_l} \zeta_{\shuffle}^T(k_1+n_1,\ldots,k_r+n_r)\right)\zeta(2k)\\
&+ \sum_{j=0}^r (-1)^{q+k_1+\ldots+k_j} \zeta(k_j,\ldots,k_1,q)\zeta_{\shuffle}^T(k_{j+1},\ldots,k_r)\\
&-2 \sum_{j=1}^r(-1)^{q+|\bfk_{[1,j)}|}\sum_{2k+m= k_j} \left(\sum_{|\bfn|=m} \binom{-q}{n_j}\prod_{l\neq j}\binom{-k_l}{n_l}\zeta_{\shuffle}^T(k_{j+1}+n_{j+1},\ldots,k_r+n_r)
\atop\times(-1)^{|\bfn_{[1,j]}|}\zeta(k_{j-1}+n_{j-1},\ldots,k_1+n_1,q+n_j)\right)\zeta(2k)=0,
\end{align*}
where $\bfn=(n_1,\ldots,n_r)\in(\mathbb{Z}_{\geqslant0})^r$ and $\zeta_{\shuffle}^T(\bfk)$ denote the shuffle regularization.
\end{thm}

\begin{proof}
For a positive multi-index $\bfk$, Ihara, Kaneko, and Zagier \cite{IKZ2006} proved the following famous regularized double shuffle theorem:
$$\rho(\zeta_*^T(\bfk))=\zeta_{\shuffle}^T(\bfk).$$
    By applying the map $\rho$ to the formula in Theorem \ref{Stuffle-Regularization}, we obtain this conclusion.
\end{proof}

\begin{exa}
    Letting $r=1$ in Theorem \ref{Shuffle-Regularization} yields the following well-known result (see \cite[Thm. 3.1]{Flajolet-Salvy}).
\begin{align*}
    &\zeta(q)\zeta_*^T(k_1)-\zeta^{\star}(k_1,q)+(-1)^q\zeta(q)\zeta_*^T(k_1)+(-1)^{q+k_1}\zeta(k_1,q)\\
    &-2\sum_{2k+m=q}\binom{-k_1}{m}\zeta_*^T(k_1+m)\zeta(2k)-2(-1)^q\sum_{2k+m=k_1}(-1)^m\binom{-q}{m}\zeta(q+m)\zeta(2k)=0
\end{align*}

\end{exa}

\begin{cor}
Let $\bfk=(k_1,\ldots,k_r)$ be an admissible multi-index, we have
$$((-1)^{k_1+\cdots+k_r}-(-1)^r)\zeta(k_1,\ldots,k_r)\equiv0\ (\rm{mod}\ \rm{products}
    ).$$
If $k_1+\cdots+k_r\not\equiv r\ (\rm{mod}\ 2)$, then $\ze(k_1,\ldots,k_r)$ can be expressed in terms of lower depth multiple zeta values.
\end{cor}
\begin{proof}
    This result comes from Theorem \ref{Stuffle-Regularization}.
\end{proof}

Similarly, by evaluating the residue of the contour integral \eqref{contourintegral-two} and applying a regularization process analogous to the one described above, one can derive parity formulas for the regularization of certain alternating multiple zeta values and multiple polylogarithms (Analytic continuation is required). The reader is encouraged to attempt this independently.

The method presented in this paper can also be applied to investigate parity formulas for related variants of multiple zeta values, such as the arbitrary cyclotomic versions of Hoffman's multiple $t$-values. In the next section, we provide parity formulas for the double shuffle regularization of cyclotomic multiple zeta values, derived using an approach analogous to the proofs of the two theorems stated above. Since the derivation follows essentially the same procedure, the details are omitted here. Interested readers are encouraged to work through the steps themselves.

It should be emphasized that in their paper \cite{CharltonHoffman-MathZ2025}, the first author of this paper and Hoffman applied the method of truncated series from Goncharov's paper \cite{Goncharov2001} by defining truncated multiple $t$-values and investigating the functional equations of their generating functions, thereby establishing a symmetry theorem for regularized multiple $t$-values \cite{H2019} (more general symmetry results for cyclotomic multiple $t$-values can also be derived). By employing analogous antipode relations (see \cite{XuZhao2020d}), a parity theorem for regularized multiple $t$-values can be similarly obtained. Moreover, we believe that the truncated series approach developed in Goncharov's paper \cite{Goncharov2001} can be extended to study symmetry or parity results for alternating multiple $M$-values (AMMVs) \cite{XuYanZhao2024}. We plan to pursue this line of research in a subsequent paper.

\section{ Parity Theorem for Cyclotomic Multiple Zeta Values}
In this finial section, we extend the conclusions obtained in the previous sections to the case of cyclotomic multiple zeta values. Since the proofs are completely analogous, we only present the results while omitting the proofs.

\begin{defn}
    For a positive multi-index $\bfk=(k_1,\ldots,k_r)$ and $\bfx=(x_1,\ldots,x_r)\ (x_i\in\mathbb{C})$, let $m_1,m_2\in\mathbb{Z}\cup\{-\infty,+\infty\}$ satisfy $m_1<m_2$. We define the following notations
    \begin{align*}
        \Li_{(m_1,m_2)}(\bfk;\bfx;s)&:=\sum_{m_1<n_1<\cdots<n_r<m_2}\frac{x_1^{n_1}\cdots x_r^{n_r}}{(n_1+s)^{k_1}\cdots (n_r+s)^{k_r}},\\
        \Li_{(m_1,m_2]}(\bfk;\bfx;s)&:=\sum_{m_1<n_1<\cdots<n_r\leqslant m_2}\frac{x_1^{n_1}\cdots x_r^{n_r}}{(n_1+s)^{k_1}\cdots (n_r+s)^{k_r}}.
    \end{align*}
    Similarly, we also define the following notations
     \begin{align*}
        \Li^{\star}_{(m_1,m_2)}(\bfk;\bfx;s)&:=\sum_{m_1<n_1\leqslant \cdots\leqslant n_r<m_2}\frac{x_1^{n_1}\cdots x_r^{n_r}}{(n_1+s)^{k_1}\cdots (n_r+s)^{k_r}},\\
        \Li^{\star}_{(m_1,m_2]}(\bfk;\bfx;s)&:=\sum_{m_1<n_1\leqslant \cdots\leqslant n_r\leqslant m_2}\frac{x_1^{n_1}\cdots x_r^{n_r}}{(n_1+s)^{k_1}\cdots (n_r+s)^{k_r}}.
    \end{align*}
    To ensure the convergence of the series, when $m_2=+\infty$ (respectively, $m_1=-\infty$), we require $k_r>1$ (respectively, $k_1>1$).
\end{defn}

In particular, for $\bfk:=(k_1,\ldots,k_r)\in (\Z_{>0})^r$ with $s\in \mathbb{C}\setminus \Z_{<0}$, the \emph{multiple Hurwitz zeta function} is defined by
\begin{align}
\zeta(\bfk;s)&:=\Li_{(0,+\infty)}(\bfk;\{1\}_r;s)=\sum_{0<n_1<\cdots<n_r} \frac1{(n_1+s)^{k_1}\cdots(n_r+s)^{k_r}},
\end{align}
and \emph{multiple polylogarithm function (of multi-variable)} is defined by
\begin{align*}
    \Li(\bfk;\bfx)&:=\Li_{(0,+\infty)}(\bfk;\bfx;0)=\sum_{0<n_1<\cdots<n_r} \frac{x_1^{n_1}\cdots x_r^{n_r}}{n_1^{k_1}\cdots n_r^{k_r}}.
\end{align*}
where $\bfx=(x_1,\ldots,x_r)\in\mathbb{C}^r,|x_i\cdots x_r|<1,1\leqslant i\leqslant n$. The multiple polylogarithm function can be analytically continued to a multi-valued meromorphic function on $\mathbb{C}^r$ (see \cite{Zhao2007d}). In general, let $\bfk=(k_1,\ldots,k_r)\in(\Z_{>0})^r$ and $\bfmu=(\mu_1,\dotsc,\mu_r)$, where $\mu_1,\dotsc,\mu_r$ are $N$th roots of unity. We can obtain the \emph{colored MZVs} of level $N$ by
\begin{equation}
\Li({\bfk};\bfmu)=\sum_{0<n_1<\cdots<n_r}
\frac{\mu_1^{n_1}\dots \mu_r^{n_r}}{n_1^{k_1} \dots n_r^{k_r}}\in \CC,
\end{equation}
which converge if $(k_r,\mu_r)\ne (1,1)$ (see \cite{YuanZh2014a} and \cite[Ch. 15]{Z2016}), in which case we call $({\bfk};\bf\mu)$ \emph{admissible}.

\begin{pro}
    Let $m_1,m_2\in\Z\cup\{-\infty,+\infty\},\ m_1<m_2$ and $n\in\mathbb{Z}$, we have the following identities
    \begin{itemize}
        \item[(1)] (Translation)
        \begin{align*}
        \Li_{(m_1,m_2)}(k_1,\ldots,k_r;x_1,\ldots,x_r;s)=(x_1\cdots x_r)^{-n}\Li_{(m_1+n,m_2+n)}(k_1,\ldots,k_r;x_1,\ldots,x_r;s-n).
    \end{align*}

    \item[(2)] (Decomposition) If $m_1<n<m_2$, then we have
    \begin{align*}
        &\Li_{(m_1,m_2)}(k_1,\ldots,k_r;x_1,\ldots,x_r;s)\\
        &=\sum_{j=0}^r\Li_{(m_1,n]}(k_1,\ldots,k_j;x_1,\ldots,x_j;s)\Li_{(n,m_2)}(k_{j+1},\ldots,k_r;x_{j+1},\ldots,x_r;s)\\
        &=\sum_{j=0}^r\Li_{(m_1,n)}(k_1,\ldots,k_j;x_1,\ldots,x_j;s)\Li_{(n,m_2)}(k_{j+1},\ldots,k_r;x_{j+1},\ldots,x_r;s)\\
        &\quad+\sum_{j=1}^r\frac{x_j^n}{(s+n)^{k_j}}\Li_{(m_1,n)}(k_1,\ldots,k_{j-1};x_1,\ldots,x_{j-1};s)\Li_{(n,m_2)}(k_{j+1},\ldots,k_r;x_{j+1},\ldots,x_r;s).
    \end{align*}

    \item[(3)] (Reflection)
    $$\Li_{(m_1,m_2)}(k_1,\ldots,k_r;x_1,\ldots,x_r;s)=(-1)^{|\bfk|}\Li_{(-m_2,-m_1)}(k_r,\ldots,k_1;x_r^{-1},\ldots,x_1^{-1};-s).$$

    \item[(4)] (Antipode identity)
    $$\sum_{j=0}^r(-1)^j\Li^{\star}_{(m_1,m_2)}(k_j,\ldots,k_1;x_j,\ldots,x_1;s)\Li_{(m_1,m_2)}(k_{j+1},\ldots,k_r;x_{j+1},\ldots,x_r;s)=0.$$

    \item[(5)] (Truncation) If $0<m_1<m_2$, then we have
    \begin{align*}
        &\Li_{(m_1,m_2)}(k_1,\ldots,k_r;x_1,\ldots,x_r;s)\\
        &=\sum_{j=0}^{r} (-1)^j \Li^\star_{(0,m_1]}(k_{j},\ldots,k_1;x_{j},\ldots,x_1;s)\cdot\Li_{(0,m_2)}(k_{j+1},k_{j+2},\ldots,k_r;x_{j+1},\ldots,x_r;s).
    \end{align*}

\item[(6)] (Expansion) If $m_1<m_2\leqslant0$ or $0\leqslant m_1<m_2$, then we have
\begin{align*}
    &\Li_{(m_1,m_2)}(k_1,\ldots,k_r;x_1,\ldots,x_r;s)\\
    &=\sum_{m=0}^{\infty}\left(\sum_{|\bfn|=m}\prod_{l=1}^r\binom{-k_l}{n_l}\Li_{(m_1,m_2)}(k_1+n_1,\ldots,k_r+n_r;x_1,\ldots,x_r)\right)s^m,
\end{align*}
where $\bfn=(n_1,\ldots,n_r)\in(\mathbb{Z}_{\geqslant0})^r$ and $|s|<1$.
    \end{itemize}
\end{pro}

We obtain that $\Li_{(m_1,m_2)}(\bfk;\bfx;s)$ has the following Laurent expansion or Taylor expansion at integer points.

\begin{thm}\label{Taextsion-Li}
Let $\bfk=(k_1,\ldots,k_r)$ be a positive multi-index, $\bfx=(x_1,\ldots,x_r)$. Let $m_1,m_2\in\mathbb{Z}\cup\{-\infty,+\infty\},\ m_1<m_2$. If $n\in\mathbb{Z}_{\geqslant-m_1}\cup\mathbb{Z}_{\leqslant-m_2}$, and $|s-n|<1$, then we have
\begin{align*}\label{Taylor-expansion-mhzf}
\Li_{(m_1,m_2)}(\bfk;\bfx;s)=\sum_{m=0}^\infty \left(\sum_{|\bfn|=m} \prod\limits_{l=1}^r \binom{-k_l}{n_l}\cdot(x_1\cdots x_r)^{-n}\Li_{(n+m_1,n+m_2)}(\bfk+\bfn;\bfx)\right)(s-n)^m,
\end{align*}
where $\bfn:=(n_1,\ldots,n_r)\in(\Z_{\geqslant0})^r$ and $|\bfn|:=n_1+\cdots+n_r$.
\end{thm}

In particular, if $n=0$ and $|s|<1$, then
\begin{align}
\Li_{(m_1,m_2)}(\bfk;\bfx;s)=\sum_{m=0}^{\infty}\left( \sum_{|\bfn|=m} \prod\limits_{l=1}^r \binom{-k_l}{n_l} \Li_{(m_1,m_2)}(\bfk;\bfx)\right)s^m.
\end{align}

\begin{thm}\label{Laextsion-Li}
Let $\bfk=(k_1,\ldots,k_r)$ be a positive multi-index, $\bfx=(x_1,\ldots,x_r)$. Let $m_1,m_2\in\mathbb{Z}\cup\{-\infty,+\infty\},\ m_1<m_2$. If $-m_2<n<-m_1$, and $|s-n|<1$, then we have
\begin{align}
\Li_{(m_1,m_2)}(\bfk;\bfx;s)=(x_1\cdots x_r)^{-n}\left(\sum_{m=0}^{\infty}a_m(s-n)^m+\sum_{j=1}^r\sum_{m=0}^{\infty}b_{m,j}(s-n)^{m-k_j}\right),
\end{align}
where
\begin{align*}
    a_m&:=\sum_{|\bfn|=m}\prod_{l=1}^r\binom{-k_l}{n_l}\sum_{j=0}^r\Li_{(n+m_1,0)}(\bfk_{[1,j]}+\bfn_{[1,j]};\bfx_{[1,j]})\Li_{(0,n+m_2)}(\bfk_{(j,r]};\bfx_{(j,r]}),\\
    b_{m,j}&:=\sum_{|\bfn|-n_j=m}\prod_{l\neq j}\binom{-k_l}{n_l}\Li_{(n+m_1,0)}(\bfk_{[1,j)}+\bfn_{[1,j)};\bfx_{[1,j)})\Li_{(0,n+m_2)}(\bfk_{(j,r]}+\bfn_{(j,r]};\bfx_{(j,r]})
\end{align*}
and $\bfn:=(n_1,\ldots,n_r)\in(\Z_{\geqslant0})^r$. The definitions of $\bfk_{[1,j]},\bfk_{(j,r]},\bfk_{[1,j)}$ are similar to the definition \ref{def2.1}. We have the following result.
\end{thm}

From now on, we always assume that $m_1,m_2\in\mathbb{Z}\cup\{-\infty,+\infty\},\ m_1<m_2$, $\bfk=(k_1,\ldots,k_r)$ be a positive multi-index, and $\bfmu=(\mu_1,\ldots,\mu_r)$, where $\mu_1,\cdots,\mu_r$ are $N$-th roots of unity.

By considering the contour integral
\begin{align}\label{contourintegral-three}
\lim_{R\rightarrow \infty}\oint_{C_R} \frac{\pi \cot(\pi s)\Li_{(m_1,m_2)}(\bfk;\bfmu;s)}{s^q}ds=0
\end{align}
and utilizing the series expansions from Theorems \ref{Taextsion-Li} and \ref{Laextsion-Li} to compute its residue, then applying a similar regularization procedure, we obtain the following parity formulas for the regularization of cyclotomic multiple zeta values.

\begin{thm}(Stuffle Regularization)
    For $q\in\Z_{>0}$ and $\bfk=(k_1,\ldots,k_r)\in (\Z_{>0})^r$, $\bfmu=(\mu_1,\ldots,\mu_r)$, with $(q,\mu_1\cdots\mu_r)\neq(1,1)$, we have
\begin{align*}
    &\sum_{j=0}^r(-1)^j\Li^{\star}(\overleftarrow{\bfk_{[1,j]}},q;\overleftarrow{\bfmu_{[1,j]}},(\mu_1\cdots \mu_r)^{-1})\Li_*^T(\bfk_{(j,r]};\bfmu_{(j,r]})\\
    &-2\sum_{2k+m=q}\left(\sum_{|\bfn|=m} \prod\limits_{l=1}^r \binom{-k_l}{n_l} \Li_*^T(\bfk+\bfn;\bfmu)\right)\zeta(2k)\\
    &+(-1)^{q+|\bfk_{[1,j]}|}\sum_{j=0}^r\Li(\overleftarrow{\bfk_{[1,j]}},q;\overleftarrow{\bfmu_{[1,j]}},\mu_1\cdots \mu_r)\Li_*^T(\bfk_{(j,r]};\bfmu_{(j,r]})\\
&\quad-2\sum_{j=1}^r(-1)^{q+|\bfk_{[1,j)}|}\sum_{2k+m= k_j} \left(\sum_{|\bfn|=m}  \binom{-q}{n_j}\prod_{l\neq j}\binom{-k_l}{n_l}\Li_*^T(\bfk_{(j,r]}+\bfn_{(j,r]};\bfmu_{(j,r]})
\atop\times(-1)^{|\bfn_{[1,j]}|}\Li(\overleftarrow{\bfk_{[1,j)}}+\overleftarrow{\bfn_{[1,j)}},q+n_j;\overleftarrow{\bfmu_{[1,j)}},\mu_1\cdots \mu_r)\right)\zeta(2k)\\
&=0,
\end{align*}
where $\bfn=(n_1,\ldots,n_r)\in(\mathbb{Z}_{\geqslant0})^r$ .
\end{thm}

Similarly, we also have shuffle regularization (The regularized double shuffle theorem for cyclotomic multiple zeta values can be found in \cite[Thm. 13.3.9]{Z2016}).

\begin{thm}(Shuffle Regularization)
    For $q\in\Z_{>0}$ and $\bfk=(k_1,\ldots,k_r)\in (\Z_{>0})^r$, $\bfmu=(\mu_1,\ldots,\mu_r)$, with $(q,\mu_1\cdots\mu_r)\neq(1,1)$, we have
\begin{align*}
    &\sum_{j=0}^r(-1)^j\Li^{\star}(\overleftarrow{\bfk_{[1,j]}},q;\overleftarrow{\bfmu_{[1,j]}},(\mu_1\cdots \mu_r)^{-1})\Li_{\shuffle}^T(\bfk_{(j,r]};\bfmu_{(j,r]})\\
    &-2\sum_{2k+m=q}\left(\sum_{|\bfn|=m} \prod\limits_{l=1}^r \binom{-k_l}{n_l} \Li_{\shuffle}^T(\bfk+\bfn;\bfmu)\right)\zeta(2k)\\
    &+(-1)^{q+|\bfk_{[1,j]}|}\sum_{j=0}^r\Li(\overleftarrow{\bfk_{[1,j]}},q;\overleftarrow{\bfmu_{[1,j]}},\mu_1\cdots \mu_r)\Li_{\shuffle}^T(\bfk_{(j,r]};\bfmu_{(j,r]})\\
&\quad-2\sum_{j=1}^r(-1)^{q+|\bfk_{[1,j)}|}\sum_{2k+m= k_j} \left(\sum_{|\bfn|=m}  \binom{-q}{n_j}\prod_{l\neq j}\binom{-k_l}{n_l}\Li_{\shuffle}^T(\bfk_{(j,r]}+\bfn_{(j,r]};\bfmu_{(j,r]})
\atop\times(-1)^{|\bfn_{[1,j]}|}\Li(\overleftarrow{\bfk_{[1,j)}}+\overleftarrow{\bfn_{[1,j)}},q+n_j;\overleftarrow{\bfmu_{[1,j)}},\mu_1\cdots \mu_r)\right)\zeta(2k)\\
&=0,
\end{align*}
where $\bfn=(n_1,\ldots,n_r)\in(\mathbb{Z}_{\geqslant0})^r$ .
\end{thm}


\medskip

{\bf Declaration of competing interest.}
The authors declares that they has no known competing financial interests or personal relationships that could have
appeared to influence the work reported in this paper.

{\bf Data availability.}
No data was used for the research described in the article.

{\bf Acknowledgments.} The authors would like to express their sincere gratitude to Steven Charlton for his valuable comments and discussions on this paper. Ce Xu is supported by the General Program of Natural Science Foundation of Anhui Province (Grant No. 2508085MA014). Ce Xu gratefully acknowledges the invitation from Professor Chengming Bai of Nankai University to the Chern Institute of Mathematics and from Professor Shaoyun Yi of Xiamen University to the Tianyuan Mathematical Center in Southeast China (TMSE). This work commenced during these visits. Jia Li is supported by the National Science Foundation of China (Grant No.12231001).
Jia Li would like to express gratitude to Professor Liang Xiao for the valuable visiting opportunity, as well as for the strong support and assistance provided during this research. Jia Li also extends thanks to the School of Mathematical Sciences at Peking University for the comfortable working environment.

\newpage



\begin{thebibliography}{99}

\bibitem{BorweinBrBr1997}
J.M.\ Borwein, D.M.\ Bradley and D.J.\ Broadhurst,
Evaluations of $k$-fold Euler/Zagier sums: a compendium of results for arbitrary $k$,
\emph{Electron.\ J.\ Combin.} \textbf{4}(2)(1997), R5.

\bibitem{Borwein-Girgensohn-1996}
J.M. Borwein and R. Girgensohn, Evaluation of triple Euler sums, \emph{Electronic J. Combin.} \textbf{3}(1996), R23.

\bibitem{Bouillot2014}
O. Bouillot, The algebra of multitangent functions, \emph{J. Algebra} \textbf{410}(2014), pp.\ 148-238.

\bibitem{BroadKre1997}
D.J. Broadhurst and D. Kreimer, D., Association of multiple zeta values with positive knots via Feynman diagrams up to 9 loops, Phys. Lett. B, 1997, 393(3/4), pp.\ 403-412.

\bibitem{BJ}
Jos\'e. Ignacio. Burgos Gil and Javier. Fres\'an, Multiple zeta values: from number to motives, http://javier.fresan.perso.math.cnrs.fr/mzv.pdf.


\bibitem{Charlton-MathAnn2025}
S. Charlton, On motivic multiple $t$-values, Saha's basis conjecture, and generators of alternating MZV's, \emph{Math. Ann.} \textbf{392}(2025), pp.\ 1995-2079.

\bibitem{CharltonHoffman-MathZ2025}
S. Charlton and M.E. Hoffman, Symmetry results for multiple $t$-values, \emph{Math. Z.} \textbf{309}(2025):75.

\bibitem{DSS2025}
A. Dixit, S. Sathyanarayana and N. Guru Sharan, Mordell-Tornheim zeta functions and functional equations for Herglotz-Zagier type functions, \emph{Adv, Math.} \textbf{473}(2025)110303.

\bibitem{Euler1}
L. Euler, ``Meditationes circa singulare serierum genus'', Novi Comm. Acad. Sci. Petropolitanae, 20 (1775), 140–186

\bibitem{Flajolet-Salvy}
P. Flajolet and B. Salvy, Euler sums and contour integral representations, \emph{Experiment. Math.} \textbf{7}(1)(1998), pp.\  15-35.

\bibitem{Goncharov2001}
A.B. Goncharov, Multiple polylogarithms and mixed Tate motives (2001), arXiv:math/0103059.

\bibitem{IKZ2006}
K. Ihara, M. Kaneko and D. Zagier, Derivation and double shuffle relations for multiple zeta values, \emph{Compos.\ Math.} \textbf{142}(2006), pp.\ 307--338.

\bibitem{Hirose2025}
M. Hirose, An explicit parity theorem for multiple zeta values via multitangent functions, \emph{Ramanujan J.} (2025)67:87.


\bibitem{H1992}
M.E. Hoffman, Multiple harmonic series, \emph{Pacific J.\ Math.} \textbf{152}(1992), pp.\ 275--290.

\bibitem{H2019}
M.E. Hoffman, An odd variant of multiple zeta values, \emph{Comm. Number Theory Phys.} \textbf{13}(2019), pp.\ 529--567.

\bibitem{KanekoYa2018}
M. Kaneko and S. Yamamoto, A new integral-series identity of multiple zeta values and regularizations, \emph{Selecta Math.} \textbf{24}(2018), pp.\ 2499--2521.


\bibitem{KanekoTs2019}
M. Kaneko and H. Tsumura, On multiple zeta values of level two, \emph{Tsukuba J.Math.} \textbf{44-2}(2020), pp.\ 213--234.

\bibitem{Li2019}
J. Li, The depth structure of motivic multiple zeta values, \emph{Math. Ann.} \textbf{374}(1-2)(2019), pp.\ 179-209.

\bibitem{Li2024}
J. Li, Unit cyclotomic multiple zeta values for $\mu_2,\mu_3$ and $\mu_4$, \emph{Adv. Math.} \textbf{438}(2024), 109466.

\bibitem{Murakami2021}
T.\ Murakami, On Hoffman's $t$-values of maximal height and generators of multiple zeta values,
\emph{Math. Ann.}, \textbf{382}(2022), pp.\ 421-458.

\bibitem{Panzer2017}
E. Panzer, The parity theorem for multiple polylogarithms, \emph{J.\ Number Theory} \textbf{172}(2017), pp.\ 93--113.

\bibitem{Rui2026}
H. Rui, Contour integrations and parity results of Hurwitz-type cyclotomic Euler sums, arXiv:2601.00035.

\bibitem{Todorov2014}
I. Todorov,  Polylogarithms and multizeta values in massless Feynman amplitudes. In Lie Theory and Its Applications in Physics; Dobrev, V., Ed.; Springer: Berlin, Heidelberg, 2014; Volume 111.

\bibitem{Tsu-2004}
H. Tsumura, Combinatorial relations for Euler-Zagier sums, \emph{Acta Arith.} \textbf{111}(2004), pp.\ 27-42.

\bibitem{Tsu-2007}
H. Tsumura, On the parity conjecture for multiple $L$-values of conductor four, \emph{Tokyo J. Math.}
\textbf{30}(2007), pp.\ 21-40.

\bibitem{Umezawa2025arxiv}
R. Umezawa, An explicit parity theorem for multiple polylogarithms, arXiv:2508.02040.

\bibitem{Xu-Wang2022}
C. Xu and W. Wang, Two variants of Euler sums, \emph{Monatsh. Math.} \textbf{199}(2022), pp.\ 431-454.

\bibitem{XuYanZhao2024}
C. Xu, L. Yan and J. Zhao, Alternating multiple mixed values: regularization, special values, parity, and dimension conjectures, \emph{Indagat. Math.} \textbf{35}(2024), pp.\ 1212-1248.

\bibitem{XuZhao2020a}
C. Xu and J. Zhao, Variants of multiple zeta values with even and
odd summation indices, \emph{Math. Zeit.} {\bf 300}(2022), pp.\ 3109-3142.

\bibitem{XuZhao2020d}
C.\ Xu and J. Zhao, Explicit relations of some variants of convoluted multiple zeta values, \emph{Ann. Math. Pura Appl.} \textbf{204}(2025), pp.\ 2065--2087.

\bibitem{YuanZh2014a}
H. Yuan and J. Zhao, Double shuffle relations of double zeta values
and double Eisenstein series of level $N$,
\emph{J.\ London Math.\ Soc.}  \textbf{92}(2)(2015), pp. \ 520--546.

\bibitem{DZ1994}
D. Zagier, Values of zeta functions and their applications, First European Congress
of Mathematics, Volume II, Birkhauser, Boston, \textbf{120}(1994), pp.\ 497--512.

\bibitem{Zhao2007d}
J. Zhao, Analytic continuation of multiple polylogarithms, \emph{Anal.\ Math.} \textbf{33}(2007), pp.\ 301--323.

\bibitem{Z2016}
J. Zhao, \emph{Multiple zeta functions, multiple polylogarithms and their special values}, Series on Number
Theory and its Applications, Vol.~12, World Scientific Publishing Co. Pte. Ltd., Hackensack, NJ, 2016.

\end{thebibliography}
\end{document}